%% file: banglogic.tex
\documentclass[a4paper,UKenglish]{lipics}
 
\usepackage{microtype}

\title{A first-order logic for string diagrams}
\titlerunning{First-order logic for string diagrams} 

\author[1]{Aleks Kissinger}
\author[2]{David Quick}
\affil[1,2]{Department of Computer Science\\
  University of Oxford, UK\\
  \texttt{\{aleks.kissinger,david.quick\}@cs.ox.ac.uk}}
\authorrunning{A. Kissinger and D. Quick} 

\Copyright{Aleks Kissinger and David Quick}


\serieslogo{}
\volumeinfo
  {Billy Editor and Bill Editors}
  {2}
  {Conference title on which this volume is based on}
  {1}
  {1}
  {1}
\EventShortName{}
\DOI{10.4230/LIPIcs.xxx.yyy.p}

\input{preamble.tex}

\usepackage{tikzfig}

\begin{document}

\maketitle


\begin{abstract}

Equational reasoning with string diagrams provides an intuitive means of proving equations between morphisms in a symmetric monoidal category. This can be extended to proofs of infinite families of equations using a simple graphical syntax called !-box notation. While this does greatly increase the proving power of string diagrams, previous attempts to go beyond equational reasoning have been largely ad hoc, owing to the lack of a suitable logical framework for diagrammatic proofs involving !-boxes. In this paper, we extend equational reasoning with !-boxes to a fully-fledged first order logic called with conjunction, implication, and universal quantification over !-boxes. This logic, called !L, is then rich enough to properly formalise an induction principle for !-boxes. We then build a standard model for !L and give an example proof of a theorem for non-commutative bialgebras using !L, which is unobtainable by equational reasoning alone.
\end{abstract}


\section{Introduction}\label{sec:introduction}

Many processes come with natural notions of parallel and sequential composition. In such cases, it is advantageous to switch from traditional term-based (i.e.~one-dimensional) syntax to the two-dimensional syntax of string diagrams. This diagrams, which consist of boxes (or various other shapes) connected by wires, form a sound and complete language for compositions of morphisms in a monoidal category~\cite{JS}. Recently, the use of string diagrams has gained much interest in a wide variety of areas, including categorical quantum mechanics~\cite{CD1,ContPhys,CDKZ}, computational linguistics~\cite{DimitriDPhil} and control theory~\cite{SobocinskiSignal,Baez2014a}.

What many of these applications have in common is they make extensive use of equational reasoning for string diagrams. That is, proofs are constructed by starting with a fixed set of diagram equations, e.g.
\ctikzfig{eq-example}
and using those to construct new equations by substitution of sub-diagrams. For example, the following is a derivation making use of the four rules above:
\ctikzfig{proof-example}
However, to prove more powerful theorems, one often needs to pass from statements about single diagrams  to entire families of diagrams and diagram equations. One way to do this, while staying within the realm of string diagrams is to use !-box notation (pronounced `bang-box notation'), introduced in~\cite{DixonDuncan2008} and formalised in~\cite{PatternGraph}. In this notation, certain sub-diagrams are wrapped in boxes, which mean `repeat this sub-diagram any number of times'. For example, suppose we considered a family of `copy' operations with 1 input and $n$ outputs. Then, if we had some other map with just a single output, we might ask that connecting it to the $n$-fold `copy' results in $n$ copies. We can represent this family of rules using !-box notation as follows:
\begin{equation}\label{eq:n-copy}
  \hfill
\beginpgfgraphicnamed{n-copy}
\InputIfFileExists{n-copy.tikz}{}{\input{./figures/n-copy.tikz}}
\endpgfgraphicnamed  \qquad \leadsto \qquad %
\beginpgfgraphicnamed{n-copy-bbox}
\InputIfFileExists{n-copy-bbox.tikz}{}{\input{./figures/n-copy-bbox.tikz}}
\endpgfgraphicnamed
  \hfill
\end{equation}
Whereas the expression on the left is informal, the expression on the right defines a family of equations without ambiguity. Formally, a !-box rule represents a set of string diagram rules obtained by \textit{instantiating} the !-box, which essentially amounts fixing the number of times to copy each !-box. For example, the instances of the !-box rule above are precisely the ones we meant to capture with the informal expression:
\ctikzfig{n-copy-bbox-inst2}
where the `blank space' in the first equation represents the monoidal unit $I$. We can even use this more expressive notation to make recursive definitions. For instance, we could recursively define the $n$-fold copy operation as a tree of binary copy operations:
\begin{equation}\label{eq:n-copy-rec}
  \hfill
\beginpgfgraphicnamed{n-copy-rec}
\InputIfFileExists{n-copy-rec.tikz}{}{\input{./figures/n-copy-rec.tikz}}
\endpgfgraphicnamed  \qquad \textrm{where} \qquad %
\beginpgfgraphicnamed{bin-copy}
\InputIfFileExists{bin-copy.tikz}{}{\input{./figures/bin-copy.tikz}}
\endpgfgraphicnamed
  \hfill\qquad
\end{equation}
Using just equational reasoning, there is no way to get from the equations in~\eqref{eq:n-copy-rec} to the $n$-fold copy equation~\eqref{eq:n-copy}. However, if we introduce an induction principle:
\ctikzfig{white_copy_induct}
we can split into a base case (zero copies of the !-box) and a step case ($n$ copies implies $n + 1$ copies). Taking the base case as given, we can prove the step case using the induction hypothesis and the rules in~\eqref{eq:n-copy-rec}:
\ctikzfig{step_case}
Unfortunately, this doesn't quite work. If we interpret $\rightarrow$ to mean `the rule on the left can be used in the proof of the rule on the right', the step case is vacuous. The rule on the right is \textit{already} an instance of the rule on the left. This is a bit like saying: $(\forall n . P n) \rightarrow (\forall n . P (n + 1))$, which is of course true for any $P$.

The problem is, when we pass to !-box notation, where single diagram rules now represent whole families of rules, our existing reasoning tools do not provide enough control over instances of rules, and how those instances interact with each other. This problem was solved for the specific case of induction in~\cite{MerryThesis} using an operation called \textit{fixing}, which essentially freezes a !-box so it can't be instantiated. However, this was introduced more as a stopgap, until a proper logic could be developed, suitable for handling conjunction, implication, and crucially universal quantification over !-boxes. In this paper, we develop that logic. With this new \textit{!-logic} in hand, we can correct our failed attempt at induction to:
\[
\hfill
\left( %
\beginpgfgraphicnamed{white_copy_base}
\InputIfFileExists{white_copy_base.tikz}{}{\input{./figures/white_copy_base.tikz}}
\endpgfgraphicnamed \, \right) \wedge
\left( {\color{blue} \forall A . } %
\beginpgfgraphicnamed{white_copy_step}
\InputIfFileExists{white_copy_step.tikz}{}{\input{./figures/white_copy_step.tikz}}
\endpgfgraphicnamed \right)
\rightarrow
\left( {\color{blue} \forall A . } %
\beginpgfgraphicnamed{n-copy-bbox}
\InputIfFileExists{n-copy-bbox.tikz}{}{\input{./figures/n-copy-bbox.tikz}}
\endpgfgraphicnamed \  \right)
\hfill\qquad
\]

In addition to giving a solid foundation for proofs constructed using !-boxes, a major motivating factor for the development of a formal logic of !-boxes is its implementation in the proof assistant Quantomatic~\cite{quanto-cade}. Currently, Quantomatic supports pure equational reasoning on string diagrams with !-boxes. The implementation of !-logic will allow it to support diagrammatic versions of all the usual trappings of a fully-featured proof assistant, such as local assumptions, goal-driven (i.e. backward) reasoning, and of course inductive proofs.

There are two essentially equivalent ways to formalise string diagrams with !-boxes: one combinatoric (as in the original formulation) and one syntactic, building on the \textit{tensor notation} for compact closed categories~\cite{NoncommBB-long}. Here we opt for the latter, as it more conveniently fits into the presentation of the logic and provides a means of elegantly representing commutative \textit{and non-commutative} generators. We begin by reviewing compact closed categories, tensor notation, and !-tensors in Section~\ref{sec:prelims}. Next, we define the concept of an instantiation, which will play a central role in the logic in Section~\ref{sec:compat}. We introduce the syntax of our logic, namely \textit{!-formulas}, in Section~\ref{sec:!-logic_formulas} and give the rules of the logic in Section~\ref{sec:rules}. We provide a semantics for !-formulas based on sets of instantiations evaluated in a compact closed category $\mathcal C$ in Section~\ref{sec:semantics}. We conclude by exhibiting a non-trivial proof involving non-commutative bialgebras, which can be done entirely within !L and diagram rewriting.


\section{Preliminaries}\label{sec:prelims}

\subsection{Compact closed categories and signatures}

Throughout this paper, we will work with \textit{compact closed categories}, i.e. symmetric monoidal categories where every object $X$ has a \textit{dual} object $X^*$ and two morphisms $\eta_X: I \to X^* \otimes X$, $\epsilon_X: X \otimes X^* \to I$ satisfying the \textit{yanking equations}:
\[
\hfill
(\epsilon_X \otimes 1_{X}) \circ (1_{X} \otimes \eta_X) = 1_{X}
\qquad\qquad
(1_{X^*} \otimes \epsilon_X) \circ (\eta_X \otimes 1_{X^*}) = 1_{X^*}
\hfill\qquad
\]
For simplicity, we will focus on \textit{strict} compact closed categories, where associativity and unitality of $\otimes$ hold on-the-nose. However, all of the concepts we will use in this paper go through virtually unmodified by Mac Lane's coherence theorem.

As string diagrams, we will depict $X$ as a wire directed upwards, and $X^*$ as a wire directed downwards. Thus $\eta_X$ and $\epsilon_X$ can be depicted as half-turns:
\[
\hfill
\eta_X \ =\ %
\beginpgfgraphicnamed{cup}
\InputIfFileExists{cup.tikz}{}{\input{./figures/cup.tikz}}
\endpgfgraphicnamed
\qquad\qquad
\epsilon_X \ =\ %
\beginpgfgraphicnamed{capp}
\InputIfFileExists{capp.tikz}{}{\input{./figures/capp.tikz}}
\endpgfgraphicnamed
\hfill\qquad
\]
which we typically call `cups' and `caps', respectively. Using this notation, the yanking equations resemble their namesake:
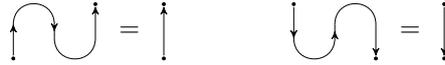
\ctikzfig{yanking}

One consequence of the inclusion of cups and caps is that we can now introduce `feedback loops', allowing us to make sense of arbitrary string diagrams, not just directed acyclic ones. A second consequence is that any map $f : X \to Y$ can be equivalently represented as a map of the form $\widetilde f : I \to X^* \otimes Y$ just by `bending' the input up to be an output:
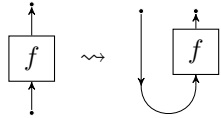
\ctikzfig{bending}
Thus, we will always assume that our generating morphisms can be written in the form $\phi: I \to X_1 \otimes X_2 \otimes \ldots \otimes X_n$ for objects $X_1,X_2,\ldots,X_n$. A morphism whose domain is the monoidal unit is called a \textit{point}.

\begin{definition}
  A \textit{compact closed signature} $\Sigma$ consists of a set $\mathcal O := \{ x, y, \ldots \}$ and a set $\mathcal M$ of pairs $(\psi, w)$, where $w$ is a word in $\{ x, x^*, y, y^*, \ldots \}$. If $\psi$ occurs precisely once in $\mathcal M$, it is said to have \textit{fixed arity}, otherwise it has \textit{variable arity}.
\end{definition}

\begin{definition}
  For a compact closed category $\mathcal C$, a \textit{valuation} $\llbracket - \rrbracket : \Sigma \to \mathcal C$ is a choice of object $X \in \textrm{ob}\,\mathcal C$ for every $x \in \mathcal O$, and a choice of point $\llbracket \psi \rrbracket : I \to X_1 \otimes X_2^* \otimes \ldots \otimes X_n$ for every $(\psi, x_1 x_2^* \ldots x_n) \in \mathcal M$.
\end{definition}

When there can be no confusion, we write pairs $(\psi, x_1 x_2^* \ldots x_n)$ also as $\psi : X_1 \otimes X_2^* \otimes \ldots X_n$. As usual, the free compact closed category $\textrm{Free}(\Sigma)$ is characterised by the universal property that any valuation lifts uniquely to functor $\llbracket - \rrbracket : \textrm{Free}(\Sigma) \to \mathcal C$ preserving all of the compact closed structure. In the next section, we will give a convenient syntactic presentation of this category.

\subsection{Tensor notation for compact closed categories}\label{sec:tensors}

From now on, we will assume that $\Sigma$ only has one object $X$, so morphisms will be maps from $I$ to monoidal products of $X$ and $X^*$.

Suppose that we have two generators in $\Sigma$, $\phi : I \to X \otimes X \otimes X^* \otimes X^* \otimes X^*$ and $\psi : I \to X \otimes X^* \otimes X^*$. Diagrammatically we will depict these generators as circular nodes with the edges ordered clockwise around the node. To avoid ambiguity we place a tick on the node between the last and first edge. We will name free edges so they can be referred to when manipulating diagrams. Hence the generators in our example (with arbitrarily named edges) are:

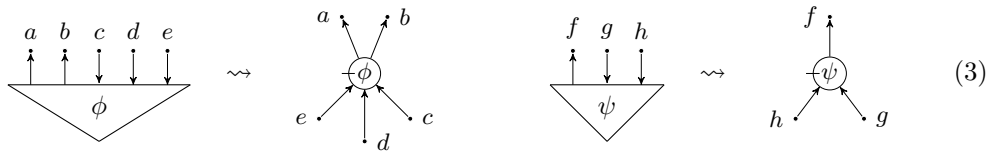
\begin{equation}\label{eq:nodes_with_ticks}
\hfill
\beginpgfgraphicnamed{phi-point}
\InputIfFileExists{phi-point.tikz}{}{\input{./figures/phi-point.tikz}}
\endpgfgraphicnamed
\quad \leadsto \quad
\beginpgfgraphicnamed{compl-fixed-arity-phi}
\InputIfFileExists{compl-fixed-arity-phi.tikz}{}{\input{./figures/compl-fixed-arity-phi.tikz}}
\endpgfgraphicnamed
\qquad\qquad
\beginpgfgraphicnamed{psi-point}
\InputIfFileExists{psi-point.tikz}{}{\input{./figures/psi-point.tikz}}
\endpgfgraphicnamed
\quad \leadsto \quad
\beginpgfgraphicnamed{compl-fixed-arity-psi}
\InputIfFileExists{compl-fixed-arity-psi.tikz}{}{\input{./figures/compl-fixed-arity-psi.tikz}}
\endpgfgraphicnamed
\hfill
\end{equation}

Now, wires connecting these dots indicate the presence of caps:
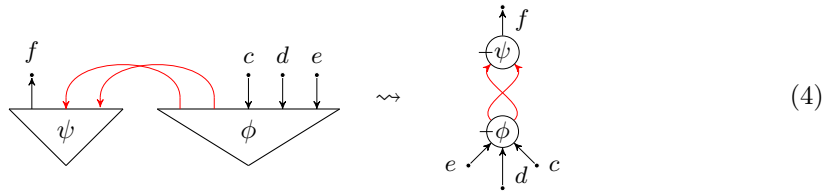
\begin{equation}\label{eq:contraction}
\hfill
\beginpgfgraphicnamed{psi-phi-contract}
\InputIfFileExists{psi-phi-contract.tikz}{}{\input{./figures/psi-phi-contract.tikz}}
\endpgfgraphicnamed
\quad \leadsto \quad
\beginpgfgraphicnamed{compose-graphs}
\InputIfFileExists{compose-graphs.tikz}{}{\input{./figures/compose-graphs.tikz}}
\endpgfgraphicnamed
\hfill
\end{equation}

To succinctly express these kinds of string diagrams syntactically, we can use \textit{tensor notation}. Here, we represent generators by writing their names, followed by a list of subscripts indicating their (named) inputs and outputs:
\[
\hfill
\phi_{\+a\+b\<c\<d\<e} \ :=\ %
\beginpgfgraphicnamed{compl-fixed-arity-phi}
\InputIfFileExists{compl-fixed-arity-phi.tikz}{}{\input{./figures/compl-fixed-arity-phi.tikz}}
\endpgfgraphicnamed
\qquad\qquad
\psi_{\+f\<g\<h} \ :=\ %
\beginpgfgraphicnamed{compl-fixed-arity-psi}
\InputIfFileExists{compl-fixed-arity-psi.tikz}{}{\input{./figures/compl-fixed-arity-psi.tikz}}
\endpgfgraphicnamed
\hfill\qquad
\]
Inputs (i.e.~outputs of type $X^*$) are represented as names with `checks' $\<a, \<b, \ldots$, whereas outputs are represented as names with `hats' $\+a, \+b, \ldots$. We combine generators into a single diagram by concatenating them, and the process of connecting generators together by caps---which we call \textit{contraction}---is indicated by repeating names:
\begin{equation}\label{eq:contraction}
\hfill
\psi_{\+f\bR\<a\<b\e}\phi_{\bR\+a\+b\e\<c\<d\<e} \ :=\ 
\beginpgfgraphicnamed{compose-graphs}
\InputIfFileExists{compose-graphs.tikz}{}{\input{./figures/compose-graphs.tikz}}
\endpgfgraphicnamed
\hfill
\end{equation}
If a name occurs once, it is called a \textit{free edgename}. If it is repeated, it is called a \textit{bound edgename}. As the name would suggest, bound edgenames have no meaning in their own right, and can be changed (a.k.a. $\alpha$-converted) at will. Hence the expressions $\psi_{\+f\bR\<a\<b\e}\phi_{\bR\+a\+b\e\<c\<d\<e}$ and $\phi_{\bR\+g\+h\e\<c\<d\<e}\psi_{\+f\bR\<g\<h\e}$ both represent~\eqref{eq:contraction}. Also, since it is the names that indicate inputs/outputs of a tensor expression, the order in which we write tensor symbols is irrelevant. So, for example, $\psi_{\+f\<a\<b} \phi_{\+a\+b\<c\<d\<e} = \phi_{\+a\+b\<c\<d\<e} \psi_{\+f\<a\<b}$.

This notation gives a simple presentation of string diagrams, and hence of morphisms in the free compact closed category over $\Sigma$. The only mismatch between tensors and morphisms in the free category is that tensors use \textit{names} to identify inputs/outputs, whereas categories use \textit{positions}. Thus, to relate the two concepts, we assume the set of edgenames contains two disjoint sets $\{ a_1, a_2, \ldots \}$ and $\{ b_1, b_2, \ldots \}$ that are totally ordered and (countably) infinite, and introduce the notion of \textit{canonically named tensors}.

\begin{definition}
  A tensor is \textit{canonically named} if its free names are $a_1, \ldots a_m, b_1, \ldots, b_n$ for some $m, n \geq 0$.
\end{definition}

We can then express a morphism in $\textrm{Free}(\Sigma)$ as a tensor whose $i$-th input is named $a_i$ and whose $j$-th output is named $b_j$. It was shown in~\cite{KissingerATS} (for the traced case) and~\cite{NoncommBB-long} (for the compact closed case) that $\textrm{Free}(\Sigma)$ is equivalent to the category whose morphisms are canonically-named tensors, with $\circ$ and $\otimes$ defined in the obvious way using renaming and contraction. This gives us an important consequence:

\begin{theorem}
  For any compact closed signature $\Sigma$, a valuation $\llbracket - \rrbracket : \Sigma \to \mathcal C$ lifts uniquely to an operation which sends canonically named tensors $G$ over $\Sigma$ to morphisms $\llbracket G \rrbracket$ in $\mathcal C$.
\end{theorem}

\subsection{!-tensors}\label{sec:bang-tensors}

As mentioned in the intro, a string diagram with !-boxes represents a family of string diagrams, where the sub-diagram in the !-box has been copied an arbitrary number of times. To formalise this, we extend the tensor syntax to include !-boxes. These extended expressions are called !-tensors. Fix disjoint, infinite sets $\edgenames$ and $\boxnames$ of \textit{edgenames} and \textit{boxnames}, respectively.

\begin{definition}\label{def:edgeterm-equiv}
The set of \textit{edgeterms} $\edgeterms$ is defined inductively as follows:
  \begin{align*}
    \bullet\; &\epsilon \in \edgeterms && \text{(empty edgeterm)} \\
    \bullet\; & \ein a, \eout a \in \edgeterms && a \in \edgenames \\
    \bullet\; &\lexp{e}^A,\rexp{e}^A \in \edgeterms && e\in\edgeterms,\; A\in\boxnames \\
    \bullet\; & e f \in \edgeterms && e,f\in\edgeterms
  \end{align*}
\end{definition}


Letting $1$ represent the empty !-tensor and $1_{\+a\<b}$ represent an identity edge with input named $b$ and output named $a$, we can define !-tensor expressions as follows:

\begin{definition}\label{def:!-tensors}
The set of all !-tensor expressions $\sgraphterms$ for a signature $\Sigma$ is defined inductively as:
  \begin{align*}
    \bullet\; &1, 1_{\+a\<b} \in \sgraphterms && a,b\in\edgenames\\
    \bullet\; &\phi_{e} \in \sgraphterms && e\in\edgeterms, \phi\in\Sigma \\
    \bullet\; &[G]^A \in \sgraphterms && G\in\sgraphterms, \; A\in\boxnames \\
    \bullet\; & G H \in \sgraphterms && G,H\in\sgraphterms
  \end{align*}
  Subject to the conditions that (F1) $\<a$ and $\+a$ must occur at most once for each edgename $a$ and (F2) $[\ldots]^A$ must occur at most once for each boxname $A$, as well as some consistency conditons for !-boxes.
\end{definition}

The remaining consistency conditions are easiest to understand in the graphical presentation of !-tensors. Sub-expressions of the form $[\ldots]^A$ are represented by wrapping a box around part of the string diagram:
\[
\hfill
\phi_{\+a}[\psi_{\<b}]^B \ :=\ %
\beginpgfgraphicnamed{bb-example-noconn}
\InputIfFileExists{bb-example-noconn.tikz}{}{\input{./figures/bb-example-noconn.tikz}}
\endpgfgraphicnamed
\hfill\qquad
\]
Edges connecting into or out of a !-box must be annotated with the !-box name and a direction, indicating whether the new edgenames should be produced to the left (anticlockwise) or to the right (clockwise) when a !-box is expanded. We indicate this direction by drawing an arc over the annotated edges:
\[
\hfill
\phi_{\lexp{\+a}^B}[\psi_{\<a}]^B \ :=\  %
\beginpgfgraphicnamed{bb-example}
\InputIfFileExists{bb-example.tikz}{}{\input{./figures/bb-example.tikz}}
\endpgfgraphicnamed \qquad\textrm{vs.}\qquad
\phi_{\rexp{\+a}^B}[\psi_{\<a}]^B \ :=\  %
\beginpgfgraphicnamed{bb-example-cw}
\InputIfFileExists{bb-example-cw.tikz}{}{\input{./figures/bb-example-cw.tikz}}
\endpgfgraphicnamed
\hfill\qquad
\]
We drop the label on the arc when it can be inferred from context. The remaining consistency conditions say that any edge connecting into or out of a !-box must have an annotation, and those annotations should respect nesting of !-boxes, as in e.g.:
\[
\hfill
\phi_{\+a \lexp{ \lexp{\<b}^B }^A} [[\phi_{\+b\<c}]^B]^A \ :=\
\beginpgfgraphicnamed{nested-ex}
\InputIfFileExists{nested-ex.tikz}{}{\input{./figures/nested-ex.tikz}}
\endpgfgraphicnamed
\hfill\qquad
\]
For a fully rigorous account of these conditions, see~\cite{NoncommBB-long}. However, the above description should suffice for the purposes of this paper, so we'll proceed to how !-tensors are instantiated. The primary instantiation operations are \textit{expand}, which produces a new copy of the contents of a !-box and \textit{kill}, which removes the !-box from the diagram:
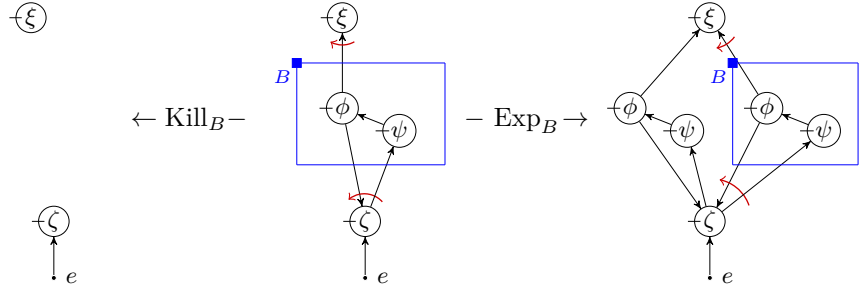
\begin{center}
\beginpgfgraphicnamed{bb-ex1-kill}
\InputIfFileExists{bb-ex1-kill.tikz}{}{\input{./figures/bb-ex1-kill.tikz}}
\endpgfgraphicnamed
  \quad$\leftarrow \Kill_B\!-$
\beginpgfgraphicnamed{bb-ex1}
\InputIfFileExists{bb-ex1.tikz}{}{\input{./figures/bb-ex1.tikz}}
\endpgfgraphicnamed
  $-\, \Exp_B\!\rightarrow$
\beginpgfgraphicnamed{bb-ex1-exp}
\InputIfFileExists{bb-ex1-exp.tikz}{}{\input{./figures/bb-ex1-exp.tikz}}
\endpgfgraphicnamed
\end{center}
These two operations suffice to produce all \textit{concrete instances}, that is all instances not involving any !-boxes, of a !-tensor. If we wish to get \textit{all} instances of a !-tensor, including those with !-boxes, we factorise expand into two additional operations: \textit{copy}, which makes a copy of the !-box and its contents, and \textit{drop}, which removes a !-box and leaves its contents behind. We can define all four of these operations recursively on !-tensor expressions. We first give the recursive cases where all four operations behave the same:
\begin{align*}
  \Op_B(G H) & := \Op_B(G) \Op_B(H) &
  \Op_B(e f) &:= \Op_B(e) \Op_B(f) \\
  \Op_B([G]^A) &:= [\Op_B(G)]^A &
  \Op_B(\rexp{e}^A) &:= \rexp{\Op_B(e)}^A \\
  \Op_B(\phi_e) &:= \phi_{\Op_B(e)} &
  \Op_B(\lexp{e}^A) &:= \lexp{\Op_B(e)}^A \\
  \Op_B(x) & : = x & &
\end{align*}
where $A \neq B$ and $x \in \{ 1, 1_{\+a\<b}, \ein a, \eout a, \epsilon \}$. The four operations are distinguished on the remaining three cases:
\begin{align*}
  \Exp_B([G]^B)  &:= [G]^B\fr(G) &
  \Kill_B([G]^B) &:= 1 \\
  \Exp_B(\rexp{e}^B)  &:= \rexp{e}^B\fr(e) &
  \Kill_B(\rexp{e}^B) &:= \epsilon \\
  \Exp_B(\lexp{e}^B)  &:= \fr(e)\lexp{e}^B &
  \Kill_B(\lexp{e}^B) &:= \epsilon \\
  \vphantom{\Exp_B} & & & \\
  \Copy_B([G]^B)  &:= [G]^B [\fr(G)]^{\fr(B)} &
  \Drop_B([G]^B)  &:= G \\
  \Copy_B(\rexp{e}^B)  &:= \rexp{e}^B \rexp{\fr(e)}^{\fr(B)} &
  \Drop_B(\rexp{e}^B)  &:= e \\
  \Copy_B(\lexp{e}^B)  &:= \lexp{\fr(e)}^{\fr(B)} \lexp{e}^B &
  \Drop_B(\lexp{e}^B)  &:= e
\end{align*}
Where $\fr$ is a function assigning fresh names to all edges and !-boxes in an expression. We occasionally write $\Exp_{B,\fr}$ and $\Copy_{B,\fr}$ to explicitly reference the freshness function of a !-box operation.

\section{Compatibility and instantiations of !-boxes}\label{sec:compat}

In Section~\ref{sec:!-logic_formulas}, we will define the formulas of !-logic. It only makes sense to combine !-tensors into single formulas if their !-boxes are compatible in some sense, so we first provide some basic notions relating to compatibility.

\begin{definition}
  If $F$ is a set and $\childof$ is a binary relation on $F$ then the pair $(F,\childof)$ is called a \textit{forest} if it forms a cycle-free directed graph where each node $A$ has at most one node $B$ s.t $A\childof B$. A forest can also be seen as a graph made up of disconnected directed trees. We write $<$ for the transitive closure and $\descof$ for the reflexive and transitive closure of $\childof$.
\end{definition}

Let $\downset{X}$ and $\upset{X}$ be the downward and upward closure of $X \subseteq F$, respectively. For a single element $A \in F$, we write $A^\leq$ for $\{ A \}^\leq$ and $A^{<}$ for $A^{\leq} \backslash A$.

\begin{definition}
  If a subset $X \subseteq F$ is both upward and downward closed (i.e. $X=\downset{X}=\upset{X}$) then we say $X$ is a \textit{component} of $(F,\childof)$. If it contains no proper sub-components, it is called a \textit{connected component}.
\end{definition}

We write $\topboxes{F} \subseteq F$ for the set of maximal elements with respect to $\leq$. Note that for $A\in\topboxes{F}$ the set $\downset{A}$ is always a connected component, and for $F$ finite, all connected components are of this form.

\begin{definition}
  Two forests $F, F'$ are said to be \textit{compatible}, written $F \compat F'$, if the intersection $F \cap F'$ is a (possibly empty) component of both $F$ and $F'$.
\end{definition}

Equivalently, $F, F'$ are compatible if and only if there exist forests $X, Y, Z$ such that $F = X \uplus Y$ and $G = Y \uplus Z$. As a consequence, the union of compatible forests is always well-defined ($F \cup F' := X \uplus Y \uplus Z$), and itself a forest. For any !-tensor, we can always associate a forest of !-boxes:

\begin{definition}
  For a !-tensor $G$, let $(\Boxes(G), \childin{G})$ be the forest of !-boxes in $G$, where $A \childin{G} B$ iff $A$ is a direct descendent of $B$. That is, $A$ is nested inside of $B$ with no intervening !-boxes.
\end{definition}

An important concept for !-tensors is that of \textit{instantiations}. These capture precisely the sequence of operations by which a !-tensor is transformed into some instance of itself. For a !-tensor $G$, an \textit{instantiation} $i$ of $G$ is a sequence of zero or more $\Exp$ and $\Kill$ operations such that $i(G)$ doesn't contain any !-boxes.

In fact, we can divorce the notion of instantiation from a particular !-tensor if we notice that instantiations make sense for any forest. For a forest $F$, and an element $B \in F$ define the $\Exp_B$ and $\Kill_B$ operations as follows:
\[
\hfill
\Exp_B(F) := F \cup \textbf{fr}(B^<) \qquad\qquad \Kill_B(F) := F \backslash B^{\leq}
\hfill\qquad
\]
where the top elements of $\textbf{fr}(B^<)$ are added as descendants of the parent of $B$ (if it has one). So, $\Kill_B$ removes $B$ and all of its children, whereas $\Exp_B$ behaves just like expanding a !-box, in that it adds a fresh copy of all of the children as siblings:
\[
\hfill
\Exp_B\left(\raisebox{7mm}{\begin{forest}
[$A$
  [$B$
    [$C$]
    [$D$]
  ]
  [$E$]
]
\end{forest}}\right) \ =\ 
\raisebox{7mm}{\begin{forest}
[$A$
  [$B$
    [$C$]
    [$D$]
  ]
  [$C'$]
  [$D'$]
  [$E$]
]
\end{forest}}
\qquad\qquad
\Kill_B\left(\raisebox{7mm}{\begin{forest}
[$A$
  [$B$
    [$C$]
    [$D$]
  ]
  [$E$]
]
\end{forest}}\right) \ =\ 
\raisebox{7mm}{\begin{forest}
[$A$
  [$E$]
]
\end{forest}}
\hfill\qquad
\]

We can now define instantiations in a way that only refers to forests:

\begin{definition}\label{def:instantiation}
  For a forest $F$, an \textit{instantiation} of $F$ is a composition $i$ of zero or more operations $\Exp_B$, $\Kill_B$ such that $B$ is in the domain of each operation and $i(F) = \{\}$. Let $\Inst(F)$ be the set of all instantiations of $F$.
\end{definition}

In particular, if $F$ is empty, $\Inst(F)$ only contains the trivial instantiation $1$. The set of instantiations for a !-tensor $G$ is then just $\Inst(\Boxes(G))$. On the other hand, $i(G)$ gives us a well-defined !-tensor for \textit{any} instantiation $i \in \Inst(F)$ when $F \compat \Boxes(G)$. This added flexibility will be important to the interpretation of !-logic formulas, where instantiations may act on many !-tensors simultaneously.

\section{!-logic formulas}\label{sec:!-logic_formulas}

In this section, we will introduce the syntax of !-logic. The atomic !-logic formulas are well-formed equations between !-tensors and generic formulas are built up from the atomic formulas using conjunction, implication, and universal quantification. 

Well-formed !-tensor equations are pairs of !-tensors with the property that any simultaneous instantiation of the LHS and RHS produces a valid equation between tensors. That is, the LHS and the RHS of any instance of the equation should have identical free edgenames for their inputs and outputs.

\begin{definition}\label{def:wfeq}
  A !-tensor equation $G = H$ is \textit{well-formed} if $G$ and $H$ have identical inputs and outputs, $\Boxes(G) \compat \Boxes(H)$, and an input $\<a$ (resp. output $\+a$) occurs in a !-box $A$ in $G$ iff it occurs in the same !-box in $H$.
\end{definition}

Note that by `$\<a$ occurs in $A$' we mean $\<a$ occurs as a sub-expression of $[\ldots]^A$, $\lexp{\ldots}^A$ or $\rexp{\ldots}^A$. The other formulas are built inductively, while maintaining the property that the sub-formulas have compatible !-boxes. To accomplish this, it is most convenient to define the set of !-formulas while simultaneously defining the operation $\Boxes(X)$ for any !-formula $X$.

\begin{definition}\label{def:bang-formulas}
  The set of \textit{!-formulas}, $\formulas{\Sigma}$, for a signature $\Sigma$ is defined inductively as:
  \begin{align*}
    \bullet\;& G=H \in \sformulas && G, H \in \graphterms_{\Sigma},\  G = H \text{ well-formed} \\
    \bullet\;& X\wedge Y \in \sformulas && X, Y \in \sformulas,\ \Boxes(X) \compat \Boxes(Y) \\
    \bullet\;& X\rightarrow Y \in \sformulas && X, Y \in \sformulas,\ \Boxes(X) \compat \Boxes(Y) \\
    \bullet\;& \quant{A} X \in \sformulas && X \in \sformulas,\  A\in\topboxes{\Boxes(X)}
  \end{align*}
    where $\Boxes(-)$ is defined recursively on !-formulas by:
  \begin{align*}
    \bullet\;& \Boxes(G=H) := \Boxes(G) \cup \Boxes(H) \\
    \bullet\;& \Boxes(X\wedge Y) := \Boxes(X)\cup\Boxes(Y) \\
    \bullet\;& \Boxes(X\rightarrow Y) := \Boxes(X)\cup\Boxes(Y) \\
    \bullet\;& \Boxes(\quant{A} X) := \Boxes(X)\backslash\downset{A}
  \end{align*}
\end{definition}

Just like one can read formulas in predicate logic as mappings from values of the free variables to truth values, one should read !-formulas as mappings from \textit{instantiations of !-boxes} to truth values. Thus, universal quantification over !-boxes states that a particular formula holds for \textit{all instantiations} involving those !-boxes. We will make this interpretation precise in Section~\ref{sec:semantics}.

One important thing to note is that universal quantification over a top-level !-box $A$ should be interpreted as quantifying over the entire \textit{connected component} $\downset{A}$. In the absence of nesting, this is the same as quantifying over individual !-boxes. However, in the presence of nesting, this restriction to only quantifying over entire components seems to be necessary for giving a consistent interpretation to !-logic formulas. This boils down to the fact that !-box operations on separate components of $\Boxes(X)$ commute, whereas arbitrary !-box operations do not.


\begin{remark}
  Note that the set $\sformulas$ in Definition~\ref{def:bang-formulas} is defined inductively by relying on a simultaneous recursive definition of $\Boxes$. This is non-circular, since the inductive steps always rely on calls to $\Boxes$ on strictly smaller formulas. Unsurprisingly, this style of definition is called \textit{induction-recursion}~\cite{IndRec}.
\end{remark}

In order to talk about instances of !-formulas, we must extend !-box operations from !-tensors to arbitrary formulas.

\begin{definition}
	For $\Op_B$ one of the operations $\Kill_B,\Exp_{B,\fr},\Copy_{B,\fr},\Drop_B$:
	\begin{itemize}
		\item $\Op_B(G=H):=\Op_B(G)=\Op_B(H)$
		\item $\Op_B(X\wedge Y):=\Op_B(X)\wedge\Op_B(Y)$
		\item $\Op_B(X\rightarrow Y):=\Op_B(X)\rightarrow\Op_B(Y)$
    \item $\Op_B(\quant{A}X):=\begin{cases}
      \quant{A}X & B\in\downset{A} \\
      \quant{A}\Op_B(X) & B\not\in\downset{A}
    \end{cases}$
  \end{itemize}
\end{definition}


\begin{theorem}
	!-box operations preserve the property of being a formula.
\end{theorem}
\begin{proof}
  We prove this using structural induction on !-formulas.
  \begin{itemize}
    \item If $G=H$ is a formula then $G$ and $H$ have the same free edges in the same !-boxes. Hence $\Op_B(G)$ and $\Op_B(H)$ have the same free edges ($a$ or $\fr(a)$ for $a$ free in $G=H$) and these are in the same !-boxes.
    \item For the next two cases we have $\Boxes(X)$ and $\Boxes(Y)$ compatible. $\Op_B$ takes the unique connected component, $S$, containing $B$ and replaces it with $\Op_B(S)$. This can only have gained fresh !-box names so $\Boxes(\Op_B(X))$ and $\Boxes(\Op_B(Y))$ are still compatible.
    \item If $B\in\downset{A}$ then the final case is trivial. If $B\not\in\downset{A}$ then the component $\downset{A}$ is not affected by $\Op_B$ so is still a component of $\Op_B(X)$.
  \end{itemize}
\end{proof}

\section{The rules of !L}\label{sec:rules}

We now define a simple logic over !-formulas, which we call !L. Our presentation is given in terms of sequents of the form: $\Gamma \vdash Y$, where $\Gamma := X_1, X_2, \ldots, X_n$ is a finite sequence of !-formulas. We will always assume in writing a sequent that all of the formulas involved have compatible !-boxes. We take the core logical rules to be those from positive intuitionistic logic with cut:
\[
\hfill
  \AxiomC{}
  \RightLabel{\scriptsize(Id)}
  \UnaryInfC{$X \vdash X$}
  \DisplayProof
\qquad
  \AxiomC{$\Gamma \vdash Y$}
  \RightLabel{\scriptsize(Weaken)}
  \UnaryInfC{$\Gamma,X \vdash Y$}
  \DisplayProof
\qquad
  \AxiomC{$\Gamma,X,Y,\Delta \vdash Z$}
  \RightLabel{\scriptsize(Perm)}
  \UnaryInfC{$\Gamma,Y,X,\Delta \vdash Z$}
  \DisplayProof
\qquad
  \AxiomC{$\Gamma,X,X \vdash Y$}
  \RightLabel{\scriptsize(Contr)}
  \UnaryInfC{$\Gamma,X \vdash Y$}
  \DisplayProof
\hfill\qquad
\]

\[
\hfill
  \AxiomC{$\Gamma \vdash X$}
  \AxiomC{$\Delta \vdash Y$}
  \RightLabel{\scriptsize($\wedge I$)}
  \BinaryInfC{$\Gamma,\Delta \vdash X\wedge Y$}
  \DisplayProof
\qquad
  \AxiomC{$\Gamma \vdash X\wedge Y$}
  \RightLabel{\scriptsize($\wedge E_1$)}
  \UnaryInfC{$\Gamma \vdash X$}
  \DisplayProof
\qquad
  \AxiomC{$\Gamma \vdash X\wedge Y$}
  \RightLabel{\scriptsize($\wedge E_2$)}
  \UnaryInfC{$\Gamma \vdash Y$}
  \DisplayProof
\hfill\qquad
\]

\[
\hfill
  \AxiomC{$\Gamma \vdash X\rightarrow Y$}
  \RightLabel{\scriptsize($\rightarrow E$)}
  \UnaryInfC{$\Gamma,X \vdash Y$}
  \DisplayProof
\qquad
  \AxiomC{$\Gamma,X \vdash Y$}
  \RightLabel{\scriptsize($\rightarrow I$)}
  \UnaryInfC{$\Gamma \vdash X\rightarrow Y$}
  \DisplayProof
\qquad
  \AxiomC{$\Gamma \vdash X$}
  \AxiomC{$\Delta,X \vdash Y$}
  \RightLabel{\scriptsize(Cut)}
  \BinaryInfC{$\Gamma,\Delta \vdash Y$}
  \DisplayProof
\hfill\qquad
\]
The rules for introducing and eliminating $\forall$ are also analogous to the usual rules. Let $\rn : \mathcal B \to \mathcal B$ be a bijective renaming function for !-boxes that is identity except on $\downset{A}$, and let $\rn(X)$ be the application of that renaming to a formula. Then:
\[
\hfill
  \AxiomC{$\Gamma \vdash \rn(X)$}
  \RightLabel{\scriptsize($\forall I$)}
  \UnaryInfC{$\Gamma \vdash \quant{A}X$}
  \DisplayProof
\qquad
  \AxiomC{$\Gamma \vdash \quant{A}X$}
  \RightLabel{\scriptsize($\forall E$)}
  \UnaryInfC{$\Gamma \vdash \rn(X)$}
  \DisplayProof
\hfill\qquad
\]
where in the case of $\forall I$ we also require that $\rn(\downset{A})$ is disjoint from $\Boxes(\Gamma)$.

To these core logical rules, we add rules capturing the fact that $=$ is an equivalence relation and a congruence:
\[
\hfill
  \AxiomC{}
  \RightLabel{\scriptsize(Refl)}
  \UnaryInfC{$\Gamma \vdash G=G$}
  \DisplayProof
  \qquad
  \AxiomC{$\Gamma \vdash G=H$}
  \RightLabel{\scriptsize(Symm)}
  \UnaryInfC{$\Gamma \vdash H=G$}
  \DisplayProof
  \qquad
  \AxiomC{$\Gamma \vdash G=H$}
  \AxiomC{$\Gamma \vdash H=K$}
  \RightLabel{\scriptsize(Trans)}
  \BinaryInfC{$\Gamma \vdash G=K$}
  \DisplayProof
\hfill\qquad
\]

\[
\hfill
  \AxiomC{$\Gamma \vdash G=H$}
  \RightLabel{\scriptsize(Box)}
  \UnaryInfC{$\Gamma \vdash [G]^A=[H]^A$}
  \DisplayProof
\qquad
  \AxiomC{$\Gamma \vdash G=H$}
  \RightLabel{\scriptsize(Prod)}
  \UnaryInfC{$\Gamma \vdash FG=FH$}
  \DisplayProof
\qquad
  \AxiomC{$\Gamma \vdash G=G'$}
  \RightLabel{\scriptsize(Ins)}
  \UnaryInfC{$\Gamma \vdash \Wk{A}{K}(G)=\Wk{A}{K}(G')$}
  \DisplayProof
\hfill\qquad
\]
where $\Wk{A}{K}$ inserts the expression $K$ into the !-box $A \in \Boxes(G)$. The last three rules allow an equation to be applied to a sub-expression. The first two rules allow us to build the context on to the outside of an equation, whereas the third one allows us to add some extra context within any !-box in an equation. These are precisely the equational reasoning rules introduced for !-tensors in~\cite{NoncommBB-long}. The only difference is we call the `weakening' operation from that paper `insertion' to avoid clash with the logical notion.


The main utility of universal quantification is to control the application !-box operations. In order to start instantiating a !-box (or one of its children), it must be under a universal quantifier:
\[
\hfill
  \AxiomC{$\Gamma \vdash \quant{A}X$}
  \RightLabel{\scriptsize($\Kill$)}
  \UnaryInfC{$\Gamma \vdash \Kill_B(X)$}
  \DisplayProof
\qquad\qquad
  \AxiomC{$\Gamma \vdash \quant{A}X$}
  \RightLabel{\scriptsize($\Exp$)}
  \UnaryInfC{$\Gamma \vdash \Exp_B(X)$}
  \DisplayProof
\hfill\qquad
\]
\[
\hfill
  \AxiomC{$\Gamma \vdash \quant{A}X$}
  \RightLabel{\scriptsize($\Drop$)}
  \UnaryInfC{$\Gamma \vdash \Drop_B(X)$}
  \DisplayProof
\qquad\qquad
  \AxiomC{$\Gamma \vdash \quant{A}X$}
  \RightLabel{\scriptsize($\Copy$)}
  \UnaryInfC{$\Gamma \vdash \Copy_B(X)$}
  \DisplayProof
\hfill\qquad
\]
where $B \leq A \in \Boxes(X)$. These rules, along with $(\forall E)$ play an analogous role to the substitution of a universally-quantified variable for an arbitrary term.

The final rule of the logic is \textit{!-box induction}, which allows us to introduce new !-boxes. For a top-level !-box $A$, we have:
\[
\hfill
  \AxiomC{$\Gamma\vdash\Kill_A(X)$}
  \AxiomC{$\Delta,X\vdash \quant{B_1} \ldots \quant{B_n} \Exp_A(X)$}
  \RightLabel{\scriptsize(Induct)}
  \BinaryInfC{$\Gamma,\Delta\vdash X$}
  \DisplayProof
\hfill\qquad
\]
where $A$ does not occur free in $\Gamma$ or $\Delta$ and $B_1$ to $B_n$ are the fresh names of children of $A$ coming from its expansion.

\section{Semantics}\label{sec:semantics}

In this section, we give a semantic interpretation for !-logic formulas using a compact closed category $\mathcal C$. For any compact closed category $\mathcal C$, a choice of valuation $\llbracket - \rrbracket : \Sigma \to \mathcal C$ of the generators in $\Sigma$ will fix a unique morphism $\llbracket G \rrbracket$ for any \textit{concrete} (i.e. !-box-free) tensor $G$. Thus $\mathcal C$ comes with an interpretation for equality between concrete tensors. From this, we can build up everything else.

For concrete tensors $G, H$, there is an obvious way to assign a truth value to the formula $G = H$:
\begin{equation}\label{eq:concrete}
  \hfill
  \llbracket G = H \rrbracket :=
\begin{cases}
  T & \textrm{ if } \llbracket G \rrbracket = \llbracket H \rrbracket \\
  F & \textrm{ otherwise}
\end{cases}
  \hfill
\end{equation}

As we first mentioned in Section~\ref{sec:!-logic_formulas}, !-logic formulas should be thought of as mappings from instantiations to truth values. Equivalently, they can be thought of as sets of instantiations: namely the set of all instantiations for which the formula holds. Applying this interpretation to atomic formulas yields the following definition:

\begin{definition}\label{def:atomic-interp}
  For an atomic !-formula $G = H$ and a valuation $\llbracket - \rrbracket : \Sigma \to \mathcal C$, we let:
  \begin{equation}\label{eq:atomic}
  \hfill
  \llbracket G = H \rrbracket = \bigg\{ i \in \Inst(\Boxes(G = H)) \ \bigg|\ \llbracket i(G) \rrbracket = \llbracket i(H) \rrbracket \bigg\}
  \hfill
  \end{equation}
\end{definition}

Concrete tensors are equal if and only if they are equal for the trivial instantiation $1$. We can interpret truth values as a special case of sets of instantiations: $T = \{ 1 \}$ and $F = \{\}$. Then, in the case of concrete tensors, \eqref{eq:atomic} reduces to \eqref{eq:concrete}.

For a forest $F$ and any $i \in \Inst(F)$, and a component $S \subseteq F$, we write $i|_S$ for the restriction of $i$ to only operations involving elements of $X$ (or fresh copies thereof). For a !-formula $X$, we write $i|_X$ for $i|_{\Boxes(X)}$. Using restrictions of instantiations, we can lift the above definition from atoms to all formulas.

\begin{definition}\label{def:interp}
  The interpretation $\sem{-}$ of a !-logic formula is defined recursively as:
  \begin{align*}
    \sem{X\wedge Y} & := \bigg\{
      \ i \in \Inst(\Boxes(X \wedge Y)) \hspace{-2.5cm}\  &
      \ \bigg| \ &
      \ i|_X \in \sem{X} \wedge i|_Y \in \sem{Y}\ 
    \bigg\} \\
    \sem{X\rightarrow Y} & := \bigg\{
      \ i \in \Inst(\Boxes(X \rightarrow Y))\hspace{-2.5cm}\  &
      \ \bigg| \ &
      \ i|_X \in \sem{X} \rightarrow i|_Y \in \sem{Y}\ 
    \bigg\} \\
    \sem{\quant{A} X} & := \bigg\{
      \ i \in \Inst(\Boxes(\quant{A} X))\hspace{-2.5cm}\  &
      \ \bigg| \ &
      \ \forall j \in \Inst(\downset{A}) \ .\  i \circ j \in \sem{X}\ 
    \bigg\}
  \end{align*}
\end{definition}

We always interpret sequents as truth values. To do so, we push all of the assumptions to the right and universally quantify over any free !-boxes:
\[
\hfill
\sem{X_1,\ldots,X_n\vdash Y} := \sem{\forall A_1 \dots \forall A_m . ((X_1\wedge\ldots\wedge X_n)\rightarrow Y)}
\hfill\qquad
\]
where $A_1, \ldots A_m$ are the free !-boxes in $X_1, \ldots, X_n, Y$.


\begin{theorem}[Soundness] \label{thm:soundness}
  If $\Gamma \vdash X$ is derivable in !L, then $\llbracket \Gamma \vdash X \rrbracket$ is true for any compact closed category $\mathcal C$.
\end{theorem}

\begin{proof}
  See Appendix~\ref{app:soundness}.
\end{proof}

The question of completeness for !L is still open. For the case of atomic !-formulas, this seems to follow straightforwardly from the fact that string diagrams (or equivalently, tensors) are sound and complete for compact closed categories. So, concrete !-tensor equations are true in all models if and only if they are identical tensors. Thus, for the case of general !-tensor equations, the problem reduces to deciding whether two !-tensors with corresponding !-boxes always have identical instances. However, once implication enters the game, we get many non-trivial formulas that hold in all models. For example, an equation with two !-boxes without edges between them always implies another equation obtained by \textit{merging} those !-boxes:
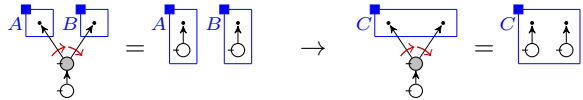
\ctikzfig{merge-example}
In this case, it is always possible to use !-box induction to prove such an implication (and many others). However, whether the rules in Section~\ref{sec:rules} suffice to get everything is a topic of continuing research.

\section{An inductive proof for non-commutative bialgebras}\label{sec:example}

In this section, we will show how !L can be used to derive highly non-trivial !-box equations using a combination of !-box induction and rewriting. Recally that a \textit{bialgebra} consists of a monoid, a comonoid, and four extra equations governing their interaction. We will extend the signature of (co)monoids to also allow for $n$-ary operations, standing for left-associated trees of multiplications and comultiplications:
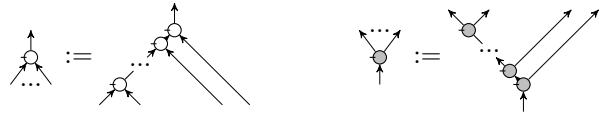
\[
\hfill
\beginpgfgraphicnamed{n-ary-tree}
\InputIfFileExists{n-ary-tree.tikz}{}{\input{./figures/n-ary-tree.tikz}}
\endpgfgraphicnamed \qquad\qquad %
\beginpgfgraphicnamed{n-ary-cotree}
\InputIfFileExists{n-ary-cotree.tikz}{}{\input{./figures/n-ary-cotree.tikz}}
\endpgfgraphicnamed
\hfill\qquad
\]
We then assume the usual (co)monoid laws, along with the definition of a higher-arity tree:
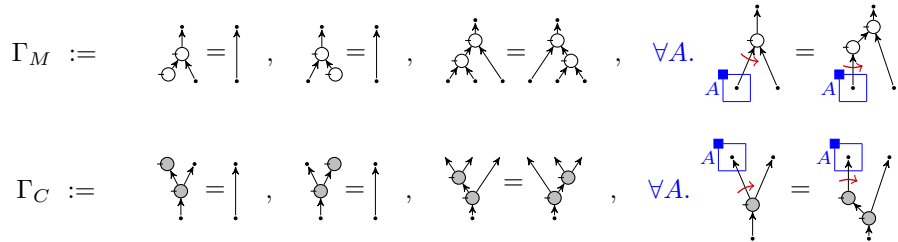
\[
\hfill
\Gamma_M \ := \qquad
\beginpgfgraphicnamed{unit-left}
\InputIfFileExists{unit-left.tikz}{}{\input{./figures/unit-left.tikz}}
\endpgfgraphicnamed \quad,\quad %
\beginpgfgraphicnamed{unit-right}
\InputIfFileExists{unit-right.tikz}{}{\input{./figures/unit-right.tikz}}
\endpgfgraphicnamed \quad,\quad
\beginpgfgraphicnamed{assoc}
\InputIfFileExists{assoc.tikz}{}{\input{./figures/assoc.tikz}}
\endpgfgraphicnamed \quad,\quad {\color{blue} \forall A . } %
\beginpgfgraphicnamed{n-mult-rec}
\InputIfFileExists{n-mult-rec.tikz}{}{\input{./figures/n-mult-rec.tikz}}
\endpgfgraphicnamed
\hfill\qquad
\]

\[
\hfill
\Gamma_C \ := \qquad
\beginpgfgraphicnamed{counit-left}
\InputIfFileExists{counit-left.tikz}{}{\input{./figures/counit-left.tikz}}
\endpgfgraphicnamed \quad,\quad %
\beginpgfgraphicnamed{counit-right}
\InputIfFileExists{counit-right.tikz}{}{\input{./figures/counit-right.tikz}}
\endpgfgraphicnamed \quad,\quad
\beginpgfgraphicnamed{coassoc}
\InputIfFileExists{coassoc.tikz}{}{\input{./figures/coassoc.tikz}}
\endpgfgraphicnamed \quad,\quad {\color{blue} \forall A . } %
\beginpgfgraphicnamed{n-copy-rec}
\InputIfFileExists{n-copy-rec.tikz}{}{\input{./figures/n-copy-rec.tikz}}
\endpgfgraphicnamed
\hfill\qquad
\]
For bialgebras, we start with these equations and add four more:
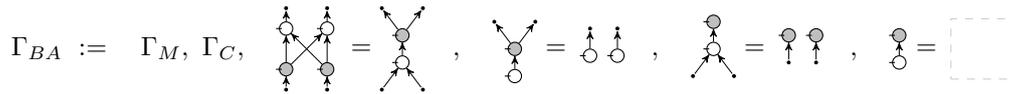
\[
\Gamma_{BA} \ := \quad \Gamma_M,\  \Gamma_C,\quad 
\beginpgfgraphicnamed{bialg1}
\InputIfFileExists{bialg1.tikz}{}{\input{./figures/bialg1.tikz}}
\endpgfgraphicnamed \quad,\quad %
\beginpgfgraphicnamed{bialg2}
\InputIfFileExists{bialg2.tikz}{}{\input{./figures/bialg2.tikz}}
\endpgfgraphicnamed \quad,\quad
\beginpgfgraphicnamed{bialg3}
\InputIfFileExists{bialg3.tikz}{}{\input{./figures/bialg3.tikz}}
\endpgfgraphicnamed \quad,\quad %
\beginpgfgraphicnamed{bialg4}
\InputIfFileExists{bialg4.tikz}{}{\input{./figures/bialg4.tikz}}
\endpgfgraphicnamed
\]

Now, we'll construct a (mostly) formal proof in !L that a tree of multiplications, followed by a tree of comultiplications is equal to a complete bipartite graph of comultiplications before multiplications. This rule generalises all 4 of the existing bialgebra rules, and can be expressed very succinctly using !-boxes:
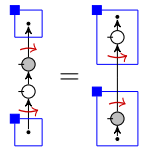
\[
\hfill
\beginpgfgraphicnamed{gen-bialg-bbox}
\InputIfFileExists{gen-bialg-bbox.tikz}{}{\input{./figures/gen-bialg-bbox.tikz}}
\endpgfgraphicnamed
\hfill\qquad
\]
To avoid massive proof trees, we will abbreviate stacks of equational reasoning rules as sequences of rewrite steps (marked with (*)'s), suppress $\forall$-intro/elim, and write (Assm) to abbreviate using an assumption. The proof from hence forth is purely graphical.

\begin{lemma}\label{lem_whitecopy}
  \quad\ctikzfig{white_copy_thm}
\end{lemma}
\begin{proof}
  \ 
  \ctikzfig{white_copy_tree}
  \[ \hfill \textrm{(*)} \quad %
\beginpgfgraphicnamed{step_case_nolab}
\InputIfFileExists{step_case_nolab.tikz}{}{\input{./figures/step_case_nolab.tikz}}
\endpgfgraphicnamed \hfill\qquad \]
\end{proof}

\begin{lemma}\label{lem_bialg}
  \quad\ctikzfig{bialg_lemma_thm}
\end{lemma}
\begin{proof}
  \ 
  \ctikzfig{bialg_lemma_tree}
  \[ \hfill \textrm{(**)} \quad %
\beginpgfgraphicnamed{bialg_lemma_step}
\InputIfFileExists{bialg_lemma_step.tikz}{}{\input{./figures/bialg_lemma_step.tikz}}
\endpgfgraphicnamed \hfill\qquad \]
\end{proof}

\begin{theorem}
  \quad\ctikzfig{gen-bialg-bbox-thm}
\end{theorem}
\begin{proof}
  \ 
  \ctikzfig{gen_bialg_tree}
  \[ \hfill \textrm{(***)} \quad %
\beginpgfgraphicnamed{bialg_bbox_step}
\InputIfFileExists{bialg_bbox_step.tikz}{}{\input{./figures/bialg_bbox_step.tikz}}
\endpgfgraphicnamed \hfill\qquad \]
\end{proof}



{\small
\bibliographystyle{akbib}
\bibliography{bibfile}
}

\newpage

\appendix

\section{Proof of soundness for !L}\label{app:soundness}

In this section, we prove Theorem~\ref{thm:soundness}, i.e. the soundness of $\sem{-}$ with respect to !L. To do so, it suffices to show that $\sem{-}$ respects each of the rules of the logic.

For $i \in \Inst(F)$ and a formula $X$ such that $\Boxes(X)$ is a component of $F$, we will write $i \vDash X$ as shorthand for $i|_X \in \sem{X}$. Using this notation, we can rewrite the interpretation as follows:
\begin{align*}
  i \vDash G = H & \iff \sem{i(G)} = \sem{i(H)} & i \in \Inst(\Boxes(G = H)) \\
  i \vDash X \wedge Y & \iff i \vDash X \wedge i \vDash Y & i \in \Inst(\Boxes(X \wedge Y)) \\
  i \vDash X \to Y & \iff i \vDash X \to i \vDash Y & i \in \Inst(\Boxes(X \to Y)) \\
  i \vDash \forall A . X & \iff \forall j \in \Inst(A^\leq) .\  i \circ j \vDash X & i \in \Inst(\Boxes(\forall A . X))
\end{align*}






Universal quantification over entire components of $\Boxes(X)$ is well-behaved for the following reason:

\begin{lemma}\label{lem:component-commute}
  For a forest $F$, let $A, B$ be elements in disinct connected components of $F$, and let $\Boxes(X) \compat F$. Then, $\Op_A(\Op'_B(X)) = \Op'_B(\Op_A(X))$ for any !-box operations $\Op_A, \Op'_B$.
\end{lemma}

\begin{proof}
  Since !-box operations recurse down to equations between !-tensors, it suffices to show that $\Op_A(\Op'_B(G = H)) = \Op'_B(\Op_A(G = H))$. Since neither $A$ nor $B$ is a child of the other, this is easy to check. The only complication is dealing with the freshness functions $\fr_A, \fr_B$ (possibly) associated with the two operations. These necessarily operate on disjoint sets of boxnames, so the only overlap might be on edgenames. However, since there is an infinite supply of fresh edgenames, it is always possible to choose new freshness functions such that $\fr_A \circ \fr_B = \fr'_B \circ \fr'_A$. Then, it is straightforward to check that $\Op_{A,\fr_A}(\Op'_{B,\fr_B}(G = H)) = \Op'_{B,\fr'_B}(\Op_{A,\fr'_A}(G = H))$.
\end{proof}

A related fact about re-ordering operations in an instantiation is that they can always be put in normal form:

\begin{lemma}\label{lem:top_op_first}
  Given an instantiation $i\in\Inst(X)$ and a top-level !-box $A\in\topboxes{X}$, $i$ can be rewritten as $i'\circ\Kill_A\circ\Exp_A^n$ where $i'\in\inst{\Kill_A\circ\Exp_A^n(X)}$.
\end{lemma}

\begin{proof}
  We need to check that operations on $A$ can always be commuted to the right, past other operations. If $B$ is not nested in $A$, this is true by Lemma~\ref{lem:component-commute}. Otherwise, $B\descof A$ and:
  \begin{itemize}
    \item If $\Op_A=\Kill_A$ then killing $A$ will erase any part of the !-formula resulting from $\Op_B$, i.e. $\Kill_A \circ \Op_B = \Kill_A$.
    \item If $\Op_A = \Exp_{A,\fr}$ then $\Exp_{A,\fr} \circ \Op_B = \Op_{\fr(B)} \circ \Op_B \circ \Exp_{A,\fr}$. In the case that $\Op_B = \Exp_B$, freshness functions on the RHS need to be chosen to produce identical names to the LHS.
  \end{itemize}
\end{proof}

\begin{notation}
  We will write $\KE^n_A$ as a shorthand for $\Kill_A\circ\Exp_A^n$.
\end{notation}

\begin{lemma}\label{lem:all-inst}
  For any !-formula $X$ and for $B_1, \ldots B_n$ the free, top-level !-boxes in $X$:
  \[
  \hfill
  \forall i \in \Inst(\Boxes(X) .\  i \vDash X \iff
  \sem{\forall B_1 \ldots \forall B_n . X} = \{ 1 \} = T
  \hfill\qquad
  \]
\end{lemma}

\begin{proof}
  First, assume the LHS, which is equivalent to $\sem{X} = \Inst(\Boxes{X})$. For any !-formula $Y$, if $B_k \in \Boxes(Y)^\top$ and $\sem{Y} = \Inst(\Boxes(Y)$, then $\sem{Y}$ contains all possible instantiations of $\Boxes(Y)$. In particular, it contains $i \circ j$ for any $i \in \Inst(\Boxes(\forall B_k . Y))$ and $j \in \Inst(B_k^\leq)$. Thus, $\sem{\forall B_k . Y} = \Inst(\Boxes(\forall B_k . Y))$. Iterating this implication, we have $\sem{\forall B_1 \ldots \forall B_n . X} = \Inst(\Boxes(\forall B_1 \ldots \forall B_n . X)) = \{ 1 \} = T$.

  Conversely, assume $\sem{\forall B_1 \ldots \forall B_n . X} = T$. Then every instantiation of the form $j = i_1 \circ i_2 \circ \ldots \circ i_n$, where the operations in $i_k$ only involve !-boxes in $B_k^\leq$ is in $\sem{X}$. But then, by Lemma~\ref{lem:component-commute}, we can freely commute !-box operations in distinct components of $\Boxes(X)$. So, in fact, every instantiation $i \in \Inst(\Boxes(X))$ is equivalent to an instantiation of the form of $j$. Then, since $j \in \sem{X}$, so is $i$.
\end{proof}

\begin{theorem}\label{thm:big-soundness}
  For any valuation $\sem{-} : \Sigma \to \mathcal C$, the rules (Id), (Weaken), (Perm), (Contr), ($\wedge I$), ($\wedge E_1$), ($\wedge E_2$), ($\to E$), ($\to I$), (Cut), ($\forall I$), ($\forall E$), (Refl), (Symm), (Trans), (Box), (Prod), (Ins), (Kill), (Exp), (Drop), (Copy), and (Induct) are sound with respect to $\sem{-}$.
\end{theorem}

\begin{proof}
  The basic structural rules just reduce to the same rules concerning instantiations. Let $K$ be the conjunction of $\Gamma$ and $K'$ the conjunction of $\Delta$ throughout. By Lemma~\ref{lem:all-inst}, to check that $\sem{\Gamma \vdash X}$ is true, it suffices to check that, for all $i \in \Inst(\Boxes(K \to X))$, $i \vDash K \to X$.

  \begin{itemize}
    \item (Ident) Fix $i \in \Inst(\Boxes(X))$. We need to show $i \in X \to X$, but this is equivalent to $i \vDash X \rightarrow i \vDash X$, which is trivially true.
    \item (Weaken) Fix $i \in \Inst(\Boxes((K \wedge X) \to Y))$ and assume $i \vDash K \to Y$. Then, if $i \vDash K \wedge X$, then $i \vDash K$. So, by assumption, $i \vDash Y$. Thus $i \vdash (K \wedge X) \to Y$.
    \item (Perm) and (Contr) follow from associativity, commutativity and idempotence of $\wedge$.
    \item ($\wedge I$) Fix $i \in \Inst(\Boxes((K \wedge K') \to (X \wedge Y)))$ and assume $i \vDash K \to X$ and $i \vDash K' \to Y$. If $i \vDash K \wedge K'$, we have $i \vDash K$ and hence $i \vDash X$. We also have $i \vDash K'$ and hence $i \vDash Y$. Thus $i \vDash X \wedge Y$.
    \item ($\wedge E1$) Fix $i \in \Inst(\Boxes(K \to X))$. Then, there exists $i' \in \Inst(\Boxes(K \to (X \wedge Y)))$ that restricts to $i$. Assume $i' \vDash K \to (X \wedge Y)$. If $i \vDash K$ then $i' \vDash K$ and hence $i' \vDash X \wedge Y$, which implies that $i' \vDash X$. So, $i \vDash X$.
    \item ($\wedge E2$) is similar to ($\wedge E1$).
    \item ($\to E$) Fix $i \in \Inst(\Boxes((K \wedge X) \to Y))$ and assume $i \vDash K \to (X \to Y)$. Then, if $i \vDash K \wedge X$ then $i \vDash K$. So, $i \vDash X \to Y$. But, since it is also the case that $i \vDash X$, $i \vDash Y$. Thus $i \vDash (K \wedge X) \to Y$.
    \item ($\to I$) is the same as ($\to E$) in reverse.
    \item (Cut) Fix $i \in \Inst(\Boxes(K \wedge K' \to X))$. Then, there exists $i' \in \Inst(\Boxes(K \to X) \cup \Boxes((K' \wedge X) \to Y))$ that restricts to $i$. Assume $i' \vDash K \to X$ and $i' \vDash (K' \wedge X) \to Y$. If $i \vDash K \wedge K'$, then $i' \vDash K \wedge K'$ so $i' \vDash K$ and $i' \vDash K'$. The former also implies that $i' \vDash X$. So, $i' \vDash K' \wedge X$ and hence $i' \vDash Y$. Finally, this implies $i \vDash Y$.
  \end{itemize}

  ($\forall I$) Fix $i \in \Inst(\Boxes(K \to \forall A . X))$. We need to show that for any $j \in \Inst(A^\leq)$, $i \circ j \vDash K \to X$. Assume without loss of generality that any !-box names on operations in $i$ are disjoint from $\rn(A^\leq)$. This is possible because $\rn(A^\leq)$ must already be disjoint from $\Boxes(\Gamma)$ (by side-condition) and it must be disjoint from $\Boxes(\forall A . X) = \Boxes(X) \backslash A^\leq$ by injectivity of $\rn$. The only other !-box names in $i$ are those introduced during instantiation, which can be freely chosen. Let $\rn(j)$ be the instantiation of $\rn(A^\leq)$ obtained by renaming operations according to $\rn$. Then, by assumption of the rule, we have $i \circ \rn(j) \vDash K \to \rn(X)$. Since $\rn$ is identity except on $A^\leq$, we have $\rn(i \circ j) \vDash \rn(K \to X)$ and thus $i \circ j \vDash K \to X$.

  ($\forall E$) Fix $i \in \Inst(\Boxes(K \to \rn(X)))$. Then suppose $i \vDash K$, then $i \vDash K$. Then, by assumption $i \vDash \forall A . X$. Let $i' = i|_{\forall A . X}$, then $i' \vDash \forall A . X$, which implies that for all $j \in \Inst(A^\leq)$, we have $i' \circ j \vDash X$. Renaming both sides yields $\rn(i' \circ j) \vDash \rn(X)$, and since $\rn$ is identity except on $A^\leq$, $i' \circ \rn(j) \vDash \rn(X)$. Now, since we are free to choose $j$, we choose it such that $(i' \circ \rn(j))|_{\rn(X)}$ is equivalent to $i|_{\rn(X)}$. Then $i \vDash \rn(X)$.

  The rules (Refl), (Symm), and (Trans) reduce to the properties of equality in $\mathcal C$. The congruence rules (Box), (Prod), and (Ins) were proven sound in~\cite{NoncommBB-long}, where the only difference here is the additional (unused) context $\Gamma$.

  (Kill) Fix $i \in \Inst(\Boxes(K \to \Kill_B(X)))$. Then if $i \vDash K$, by assumption $i \vDash \forall A . X$. Since $B \leq A$ does not occur free in $\forall A . X$, $i \circ \Kill_B \vDash \forall A . X$. For $i' = i|_{\forall A . X}$, choose $j \in \Inst(A^\leq)$ such that $(i' \circ j)|_X$ is equivalent to $(i \circ \Kill_B)|_X$. Then, $i' \circ j \vDash X$, so $i \circ \Kill_B \vDash X$, and $i \vDash \Kill_B(X)$. (Exp) is similar.

  (Copy) and (Drop) are is also similar. However, when we choose $j \in \Inst(A^\leq)$ such that $(i' \circ j)|_X$ is equivalent to $(i \circ \Copy_B)|_X$ or $(i \circ \Copy_B)|_X$, we make use of the fact that instantiations involving $\Copy/\Drop$ can always be reduced to a normal form which only includes $\Exp$ and $\Kill$. This was proven in~\cite{NoncommBB-long}.

  Finally, we prove the (Induct) rule. For any top-level !-box $A$, Lemma~\ref{lem:top_op_first} says that we can write any instantiation $i$ equivalently as $j \circ \KE^n_A$, where $j$ doesn't contain $A$. Thus, we will show that, for all $n$, and all instantiations of the form $i := j \circ \KE^n_A$, $i \vDash (K \wedge K') \to X$. We proceed by induction on $n$. 

  For the base case, $i = j \circ \Kill_A$. If $i \vDash K$, then since $K$ doesn't contain $A$, $i \vDash K$ implies $j \vDash K$. So, by the first premise $j \vDash \Kill_A(X)$. Thus $j \circ \Kill_A \vDash X$, as required. For the step case, assume that for all instantiations of $(K \wedge K') \to X$ of the form $i := j \circ \KE^n_A$, $i \vDash (K \wedge K') \to X$. We need to show for all $i' := j \circ \KE^{n+1}_A$, $i' \vDash (K \wedge K') \to X$. If $i' \vDash K \wedge K'$, then $i' \vDash K'$. Then, since $K$ doesn't contain $A$, $i \vDash K'$. Combining this with the induction hypothesis yields $i \vDash \forall B_1 \ldots \forall B_m . \Exp_A(X)$. Thus, for any instantiation $k$ of the $B_1^\leq, \ldots, B_m^\leq$, $i \circ k \vDash \Exp_A(X)$. So, $i \circ k \circ \Exp_A \vDash X$. $i'$ is equivalent to $i \circ k \circ \Exp_A$ for some $i, k$, so $i' \vDash X$.
\end{proof}

Soundness of !L with respect to $\sem{-}$ then follows from Theorem~\ref{thm:big-soundness}.

\end{document}

%% file: preamble.tex
\usepackage{tikzfig}
\input{tikzstyles.tex}

\usepackage{noncommgraph}

\usepackage{amsmath}
\usepackage{amssymb}
\usepackage{stmaryrd}
\usepackage{amsthm}
\usepackage{xspace}
\usepackage{hyperref}
\usepackage{bussproofs}
\usepackage{bpextra}
\usepackage{forest}
\usepackage{graphicx}
\graphicspath{{figures/}}

\usepackage{algpseudocode}
\usepackage{algorithm}
\algblock[Switch]{Switch}{EndSwitch}
\algblock[Case]{Case}{EndCase}
\algtext*{EndCase}

\theoremstyle{definition}\newtheorem{notation}[theorem]{Notation}

\input{defs.tex}

\def\bR{\begin{color}{red}}
\def\bB{\begin{color}{blue}}
\def\bM{\begin{color}{magenta}}
\def\bC{\begin{color}{cyan}}
\def\bW{\begin{color}{white}}
\def\bBl{\begin{color}{black}}
\def\bG{\begin{color}{green}}
\def\bY{\begin{color}{yellow}}
\def\bDG{\begin{color}{green!70!black}}
\def\e{\end{color}}

\usepackage{wrapfig}
  {\gdef\scalefactor{1.0}\begin{center}\proofSkipAmount \leavevmode}%
  {\scalebox{\scalefactor}{\DisplayProof}\proofSkipAmount \end{center} }

%% file: tikzstyles.tex
\tikzstyle{every picture}=[baseline=-0.25em,scale=0.6]
\tikzstyle{small box}=[rectangle,draw,fill=white,minimum width=3mm,minimum height=3mm]
\tikzstyle{empty diagram}=[draw=gray!40!white,dashed,shape=rectangle,minimum width=8mm,minimum height=8mm]

\tikzstyle{dotpic}=[scale=0.6]
\tikzstyle{directeds}=[every to/.style={directed}]
\tikzstyle{dot graph}=[shorten <=-0.1mm,shorten >=-0.1mm,scale=0.6]
\tikzstyle{digraph}=[-latex]
\tikzstyle{plot point}=[circle,fill=black,minimum width=2mm,inner sep=0]
\tikzstyle{string graph}=[scale=0.6]
\tikzstyle{sg directed}=[-stealth]
\tikzstyle{rewrite edge}=[-open triangle 45]
\tikzstyle{sg bold directed}=[-stealth,thick,shorten >=-1pt]
\tikzstyle{sg vertex}=[circle,minimum width=2.2mm,fill=white,draw=black,inner sep=0mm]
\tikzstyle{labelled sg vertex}=[circle,minimum width=7mm,fill=white,draw=black,inner sep=0mm]
\tikzstyle{sg grey vertex}=[sg vertex,fill=gray!30!white]
\tikzstyle{sg black vertex}=[sg vertex,fill=black]
\tikzstyle{sg bold vertex}=[circle,minimum width=2.2mm,fill=white,draw=black,very thick,inner sep=0mm]
\tikzstyle{sg wire vertex}=[circle,minimum width=1mm,fill=black,inner sep=0mm]

\tikzstyle{tick vertex}=[rectangle,fill=black,minimum height=1mm,minimum width=2.5mm,inner sep=0mm]

\tikzstyle{braceedge}=[decorate,decoration={brace,amplitude=2mm,raise=-1mm}]
\tikzstyle{small braceedge}=[decorate,decoration={brace,amplitude=1mm,raise=-1mm}]
\tikzstyle{left hook arrow}=[left hook-latex]
\tikzstyle{right hook arrow}=[right hook-latex]

\tikzstyle{dot}=[inner sep=0.7mm,minimum width=0pt,minimum height=0pt,fill=black,draw=black,shape=circle]
\tikzstyle{white dot}=[dot,fill=white]
\tikzstyle{alt white dot}=[white dot,label={[xshift=2.9mm,yshift=-0.1mm]left:$\cdot$}]
\tikzstyle{gray dot}=[dot,fill=gray!50]
\tikzstyle{box vertex}=[draw=black,rectangle]
\tikzstyle{whitebg}=[fill=white,inner sep=2pt]
\tikzstyle{graph state vertex}=[sg vertex,fill=black]

\tikzstyle{wide point}=[fill=white,draw=black,shape=isosceles triangle,shape border rotate=90,isosceles triangle stretches=true,inner sep=1pt,minimum width=1.5cm,minimum height=5mm]
\tikzstyle{wide copoint}=[fill=white,draw=black,shape=isosceles triangle,shape border rotate=-90,isosceles triangle stretches=true,inner sep=1pt,minimum width=1.5cm,minimum height=5mm]
\tikzstyle{symm}=[ultra thick,shorten <=-1mm,shorten >=-1mm]

\tikzstyle{square box}=[rectangle,fill=white,draw=black,minimum height=6mm,minimum width=6mm]
\tikzstyle{square gray box}=[rectangle,fill=gray!30,draw=black,minimum height=6mm,minimum width=6mm]
\tikzstyle{point}=[regular polygon,regular polygon sides=3,draw=black,scale=0.75,inner sep=-0.5pt,minimum width=7mm,fill=white]
\tikzstyle{copoint}=[point,regular polygon rotate=180,fill=white]
\tikzstyle{gray point}=[point,fill=gray!40!white]
\tikzstyle{gray copoint}=[copoint,fill=gray!40!white]

\tikzstyle{open graph}=[baseline=-0.25em]
\tikzstyle{greybg}=[background rectangle/.style={fill=black!5,draw=black!30,rounded corners=1ex}, show background rectangle]

\tikzstyle{edge point}=[circle,minimum width=1mm,fill=black,inner sep=0mm]
\tikzstyle{vertex point}=[circle,minimum width=2.2mm,fill=white,draw=black,inner sep=0mm]
\tikzstyle{gray vertex point}=[circle,minimum width=2.2mm,fill=gray!30!white,draw=black,inner sep=0mm]
\tikzstyle{edge label}=[inner sep=2pt, font=\small]
\tikzstyle{on edge label}=[fill=white, font=\footnotesize, inner sep=1 pt]

\newcommand{\edgearrow}{{\arrow[black]{>}}}
\newcommand{\edgetick}{{\arrow[black,scale=0.7,very thick]{|}}}
\tikzstyle{directed}=[postaction=decorate,decoration={markings, mark=at position 0.55 with \edgearrow}]
\tikzstyle{medium directed}=[postaction=decorate,decoration={markings, mark=at position 0.75 with \edgearrow}]
\tikzstyle{short directed}=[->]
\tikzstyle{halfedge}=[-)]
\tikzstyle{other halfedge}=[(-]
\tikzstyle{freeedge}=[(-)]
\tikzstyle{white edge}=[line width=5pt,white]
\tikzstyle{tick}=[postaction=decorate,decoration={markings, mark=at position 0.5 with \edgetick}]
\tikzstyle{small map edge}=[|-latex, gray!60!blue, shorten <=0.9mm, shorten >=0.5mm]
\tikzstyle{thick dashed edge}=[very thick,dashed,gray!40]
\tikzstyle{dashed edge}=[densely dotted,thick]
\tikzstyle{map edge}=[|-latex,very thick, gray!40, shorten <=1mm, shorten >=0.5mm]
\tikzstyle{tickedge}=[postaction=decorate,
  decoration={markings, mark=at position 0.5 with \edgetick}]
\tikzstyle{dirtickedge}=[postaction=decorate,
  decoration={markings, mark=at position 0.5 with \edgetick},
  decoration={markings, mark=at position 0.85 with \edgearrow}]
\tikzstyle{dirdoubletickedge}=[postaction=decorate,
  decoration={markings, mark=at position 0.4 with \edgetick},
  decoration={markings, mark=at position 0.6 with \edgetick},
  decoration={markings, mark=at position 0.85 with \edgearrow}]

\tikzstyle{arrs}=[-latex,font=\small,auto]
\tikzstyle{arrow plain}=[arrs]
\tikzstyle{arrow dashed}=[dashed,arrs]
\tikzstyle{arrow bold}=[very thick,arrs]
\tikzstyle{arrow hide}=[draw=white!0,-]
\tikzstyle{arrow reverse}=[latex-]
\tikzstyle{cdnode}=[]

\tikzstyle{cnot}=[fill=white,shape=circle,inner sep=-1.4pt]
\tikzstyle{bang box}=[draw=black,dashed,minimum height=12mm,minimum width=12mm,fill=gray!20]
\tikzstyle{wire label}=[font=\footnotesize, auto]

%% file: defs.tex
\tikzstyle{cdiag}=[matrix of math nodes, row sep=3em, column sep=3em, text height=1.5ex, text depth=0.25ex,inner sep=0.5em]
\tikzstyle{arrow above}=[transform canvas={yshift=0.5ex}]
\tikzstyle{arrow below}=[transform canvas={yshift=-0.5ex}]



%% file: figures/n-copy.tikz
\begin{tikzpicture}[dotpic]
	\begin{pgfonlayer}{nodelayer}
		\node [style=white dot] (0) at (-1.5, -1.25) {};
		\node [style=gray dot] (1) at (-1.5, 0) {};
		\node [style=wire] (2) at (-2.5, 1.25) {};
		\node [style=wire] (3) at (-0.5, 1.25) {};
		\node [style=none] (4) at (-1.5, 1) {...};
		\node [style=none] (5) at (0.5, 0) {$=$};
		\node [style=wire] (6) at (2, 1.25) {};
		\node [style=none] (7) at (2.75, 0.75) {...};
		\node [style=wire] (8) at (3.5, 1.25) {};
		\node [style=white dot] (9) at (2, 0) {};
		\node [style=white dot] (10) at (3.5, 0) {};
	\end{pgfonlayer}
	\begin{pgfonlayer}{edgelayer}
		\draw [style=directed] (0) to (1);
		\draw [style=directed, bend left=15, looseness=1.00] (1) to (2);
		\draw [style=directed, bend right=15, looseness=1.00] (1) to (3);
		\draw [style=directed] (9) to (6);
		\draw [style=directed] (10) to (8);
	\end{pgfonlayer}
\end{tikzpicture}

%% file: figures/n-copy-bbox.tikz
\begin{tikzpicture}[dotpic]
	\begin{pgfonlayer}{nodelayer}
		\node [style=white] (0) at (2, 0) {};
		\node [style=none] (1) at (1.5, -0.5) {};
		\node [style=wire] (2) at (-1.25, 1.25) {};
		\node [style=bbox, label={$A$}] (3) at (1.5, 1.75) {};
		\node [style=none] (4) at (-1.75, 0.75) {};
		\node [style=none] (5) at (-0.75, 1.75) {};
		\node [style=white] (6) at (-1.25, -1.5) {};
		\node [style=bbox, label={$A$}] (7) at (-1.75, 1.75) {};
		\node [style=wire] (8) at (2, 1.25) {};
		\node [style=none] (9) at (2.5, 1.75) {};
		\node [style=none] (10) at (2.5, -0.5) {};
		\node [style=none] (11) at (-0.75, 0.75) {};
		\node [style=gray] (12) at (-1.25, -0.25) {};
		\node [style=none] (13) at (0.25, 0) {$=$};
	\end{pgfonlayer}
	\begin{pgfonlayer}{edgelayer}
		\draw [style=directed] (6) to (12);
		\draw [style=directed, arcout={{{}{0}{45}{1.5mm}{1}}}] (12) to (2);
		\draw [style=directed] (0) to (8);
		\draw [style=boxedge] (7) to (5.center);
		\draw [style=boxedge] (5.center) to (11.center);
		\draw [style=boxedge] (11.center) to (4.center);
		\draw [style=boxedge] (4.center) to (7);
		\draw [style=boxedge] (3) to (9.center);
		\draw [style=boxedge] (9.center) to (10.center);
		\draw [style=boxedge] (10.center) to (1.center);
		\draw [style=boxedge] (1.center) to (3);
	\end{pgfonlayer}
\end{tikzpicture}

%% file: figures/n-copy-rec.tikz
\begin{tikzpicture}[dotpic]
	\begin{pgfonlayer}{nodelayer}
		\node [style=wire] (0) at (-2.5, 1.25) {};
		\node [style=none] (1) at (-3, 0.75) {};
		\node [style=none] (2) at (-2, 1.75) {};
		\node [style=wire] (3) at (-1.75, -1.75) {};
		\node [style=bbox, label={$A$}] (4) at (-3, 1.75) {};
		\node [style=none] (5) at (-2, 0.75) {};
		\node [style=gray] (6) at (-1.75, -0.5) {};
		\node [style=none] (7) at (0, 0) {$=$};
		\node [style=wire] (8) at (-1, 1.25) {};
		\node [style=none] (9) at (2.25, 1.75) {};
		\node [style=wire] (10) at (2.5, -1.75) {};
		\node [style=gray dot] (11) at (2.5, -1) {};
		\node [style=none] (12) at (2.25, 0.75) {};
		\node [style=bbox, label={$A$}] (13) at (1.25, 1.75) {};
		\node [style=wire] (14) at (1.75, 1.25) {};
		\node [style=none] (15) at (1.25, 0.75) {};
		\node [style=wire] (16) at (3.25, 1.25) {};
		\node [style=gray] (17) at (1.75, -0.25) {};
	\end{pgfonlayer}
	\begin{pgfonlayer}{edgelayer}
		\draw [style=directed] (3) to (6);
		\draw [style=directed, arcout={{{}{0}{45}{1.5mm}{1}}}] (6) to (0);
		\draw [style=boxedge] (4) to (2.center);
		\draw [style=boxedge] (2.center) to (5.center);
		\draw [style=boxedge] (5.center) to (1.center);
		\draw [style=boxedge] (1.center) to (4);
		\draw [style=directed] (6) to (8);
		\draw [style=directed] (10) to (11);
		\draw [style=boxedge] (13) to (9.center);
		\draw [style=boxedge] (9.center) to (12.center);
		\draw [style=boxedge] (12.center) to (15.center);
		\draw [style=boxedge] (15.center) to (13);
		\draw [style=directed] (11) to (16);
		\draw [style=directed] (11) to (17);
		\draw [style=directed, arcout={{{}{0}{45}{1.5mm}{1}}}] (17) to (14);
	\end{pgfonlayer}
\end{tikzpicture}

%% file: figures/bin-copy.tikz
\begin{tikzpicture}[dotpic]
	\begin{pgfonlayer}{nodelayer}
		\node [style=wire] (0) at (-2.5, 1.25) {};
		\node [style=white dot] (1) at (-1.75, -1.25) {};
		\node [style=gray dot] (2) at (-1.75, 0) {};
		\node [style=none] (3) at (0, 0) {$=$};
		\node [style=wire] (4) at (-1, 1.25) {};
		\node [style=wire] (5) at (1.75, 1.25) {};
		\node [style=white dot] (6) at (1.75, 0) {};
		\node [style=white dot] (7) at (3, 0) {};
		\node [style=wire] (8) at (3, 1.25) {};
	\end{pgfonlayer}
	\begin{pgfonlayer}{edgelayer}
		\draw [style=directed] (1) to (2);
		\draw [style=directed] (2) to (0);
		\draw [style=directed] (2) to (4);
		\draw [style=directed] (6) to (5);
		\draw [style=directed] (7) to (8);
	\end{pgfonlayer}
\end{tikzpicture}

%% file: figures/white_copy_base.tikz
\begin{tikzpicture}[dotpic]
	\begin{pgfonlayer}{nodelayer}
		\node [style=gray] (0) at (-1, 0.5) {};
		\node [style=white] (1) at (-1, -0.5) {};
		\node [style=none] (2) at (0, 0) {$=$};
		\node [style=empty diagram] (3) at (2, 0) {};
	\end{pgfonlayer}
	\begin{pgfonlayer}{edgelayer}
		\draw [style=directed] (1) to (0);
	\end{pgfonlayer}
\end{tikzpicture}

%% file: figures/white_copy_step.tikz
\begin{tikzpicture}[dotpic]
	\begin{pgfonlayer}{nodelayer}
		\node [style=wire] (0) at (-5.75, 1.25) {};
		\node [style=none] (1) at (-6.25, 0.75) {};
		\node [style=bbox, label={$A$}] (2) at (-2.75, 1) {};
		\node [style=bbox, label={$A$}] (3) at (-6.25, 1.75) {};
		\node [style=gray] (4) at (-5.75, 0) {};
		\node [style=white] (5) at (-5.75, -1) {};
		\node [style=none] (6) at (-5.25, 1.75) {};
		\node [style=white] (7) at (-2.25, -0.25) {};
		\node [style=none] (8) at (-1.75, -0.75) {};
		\node [style=none] (9) at (-1.75, 1) {};
		\node [style=none] (10) at (-4.25, 0) {$=$};
		\node [style=none] (11) at (-5.25, 0.75) {};
		\node [style=none] (12) at (-2.75, -0.75) {};
		\node [style=wire] (13) at (-2.25, 0.75) {};
		\node [style=none] (14) at (0, 0) {$\rightarrow$};
		\node [style=white] (15) at (2.75, -1) {};
		\node [style=gray] (16) at (2.75, 0) {};
		\node [style=none] (17) at (4.25, 0) {$=$};
		\node [style=wire] (18) at (2, 1.25) {};
		\node [style=wire] (19) at (6.25, 0.75) {};
		\node [style=white] (20) at (6.25, -0.25) {};
		\node [style=wire] (21) at (3.5, 1.25) {};
		\node [style=white] (22) at (7.5, -0.25) {};
		\node [style=wire] (23) at (7.5, 0.75) {};
		\node [style=bbox, label={$A$}] (24) at (5.75, 1) {};
		\node [style=none] (25) at (1.5, 0.75) {};
		\node [style=none] (26) at (2.5, 0.75) {};
		\node [style=none] (27) at (2.5, 1.75) {};
		\node [style=none] (28) at (5.75, -0.75) {};
		\node [style=none] (29) at (6.75, -0.75) {};
		\node [style=none] (30) at (6.75, 1) {};
		\node [style=bbox, label={$A$}] (31) at (1.5, 1.75) {};
	\end{pgfonlayer}
	\begin{pgfonlayer}{edgelayer}
		\draw [style=directed] (5) to (4);
		\draw [style=right arcout] (4) to (0);
		\draw [style=directed] (7) to (13);
		\draw [style=directed] (15) to (16);
		\draw [style=right arcout] (16) to (18);
		\draw [style=directed] (20) to (19);
		\draw [style=directed] (16) to (21);
		\draw [style=directed] (22) to (23);
		\draw [style=boxedge] (3) to (6.center);
		\draw [style=boxedge] (6.center) to (11.center);
		\draw [style=boxedge] (11.center) to (1.center);
		\draw [style=boxedge] (1.center) to (3);
		\draw [style=boxedge] (2) to (9.center);
		\draw [style=boxedge] (9.center) to (8.center);
		\draw [style=boxedge] (8.center) to (12.center);
		\draw [style=boxedge] (12.center) to (2);
		\draw [style=boxedge] (27.center) to (26.center);
		\draw [style=boxedge] (29.center) to (28.center);
		\draw [style=boxedge] (31) to (27.center);
		\draw [style=boxedge] (30.center) to (29.center);
		\draw [style=boxedge] (28.center) to (24);
		\draw [style=boxedge] (26.center) to (25.center);
		\draw [style=boxedge] (25.center) to (31);
		\draw [style=boxedge] (24) to (30.center);
	\end{pgfonlayer}
\end{tikzpicture}

%% file: figures/cup.tikz
\begin{tikzpicture}
	\begin{pgfonlayer}{nodelayer}
		\node [style=wire] (0) at (-0.75, 0.25) {};
		\node [style=wire] (1) at (0.75, 0.25) {};
	\end{pgfonlayer}
	\begin{pgfonlayer}{edgelayer}
		\draw [style=directed, in=-90, out=-90, looseness=1.75] (0) to (1);
	\end{pgfonlayer}
\end{tikzpicture}

%% file: figures/capp.tikz
\begin{tikzpicture}
	\begin{pgfonlayer}{nodelayer}
		\node [style=wire] (0) at (-0.75, -0.5) {};
		\node [style=wire] (1) at (0.75, -0.5) {};
	\end{pgfonlayer}
	\begin{pgfonlayer}{edgelayer}
		\draw [style=directed, in=90, out=90, looseness=1.75] (0) to (1);
	\end{pgfonlayer}
\end{tikzpicture}

%% file: figures/phi-point.tikz
\begin{tikzpicture}
	\begin{pgfonlayer}{nodelayer}
		\node [style=none] (0) at (0, -0.75) {$\phi$};
		\node [style=none] (1) at (-2, -0.25) {};
		\node [style=none] (2) at (2, -0.25) {};
		\node [style=none] (3) at (0, -1.5) {};
		\node [style=wire, label={above:$a$}] (4) at (-1.5, 0.5) {};
		\node [style=wire, label={above:$b$}] (5) at (-0.75, 0.5) {};
		\node [style=wire, label={above:$c$}] (6) at (0, 0.5) {};
		\node [style=wire, label={above:$d$}] (7) at (0.75, 0.5) {};
		\node [style=wire, label={above:$e$}] (8) at (1.5, 0.5) {};
		\node [style=none] (9) at (-0.75, -0.25) {};
		\node [style=none] (10) at (1.5, -0.25) {};
		\node [style=none] (11) at (0, -0.25) {};
		\node [style=none] (12) at (-1.5, -0.25) {};
		\node [style=none] (13) at (0.75, -0.25) {};
	\end{pgfonlayer}
	\begin{pgfonlayer}{edgelayer}
		\draw (1.center) to (3.center);
		\draw (3.center) to (2.center);
		\draw (2.center) to (1.center);
		\draw [style=directed] (12.center) to (4);
		\draw [style=directed] (9.center) to (5);
		\draw [style=directed] (6.center) to (11);
		\draw [style=directed] (7) to (13.center);
		\draw [style=directed] (8) to (10.center);
	\end{pgfonlayer}
\end{tikzpicture}

%% file: figures/compl-fixed-arity-phi.tikz
\begin{tikzpicture}
	\begin{pgfonlayer}{nodelayer}
		\node [style=arbi] (0) at (0, 0) {$\phi$};
		\node [style=wire, label={left:$a$}] (1) at (-0.5, 1.25) {};
		\node [style=wire, label={right:$b$}] (2) at (0.5, 1.25) {};
		\node [style=wire, label={right:$c$}] (3) at (1, -1) {};
		\node [style=wire, label={right:$d$}] (4) at (0, -1.5) {};
		\node [style=wire, label={left:$e$}] (5) at (-1, -1) {};
	\end{pgfonlayer}
	\begin{pgfonlayer}{edgelayer}
		\draw [style=directed] (0) to (1);
		\draw [style=directed] (0) to (2);
		\draw [style=directed] (3) to (0);
		\draw [style=directed] (4) to (0);
		\draw [style=directed] (5) to (0);
	\end{pgfonlayer}
\end{tikzpicture}

%% file: figures/psi-point.tikz
\begin{tikzpicture}
	\begin{pgfonlayer}{nodelayer}
		\node [style=none] (0) at (0, -0.75) {$\psi$};
		\node [style=none] (1) at (-1.25, -0.25) {};
		\node [style=none] (2) at (1.25, -0.25) {};
		\node [style=none] (3) at (0, -1.5) {};
		\node [style=wire, label={above:$f$}] (4) at (-0.75, 0.5) {};
		\node [style=wire, label={above:$g$}] (5) at (0, 0.5) {};
		\node [style=wire, label={above:$h$}] (6) at (0.75, 0.5) {};
		\node [style=none] (7) at (-0.75, -0.25) {};
		\node [style=none] (8) at (0, -0.25) {};
		\node [style=none] (9) at (0.75, -0.25) {};
	\end{pgfonlayer}
	\begin{pgfonlayer}{edgelayer}
		\draw (1.center) to (3.center);
		\draw (3.center) to (2.center);
		\draw (2.center) to (1.center);
		\draw [style=directed] (7.center) to (4);
		\draw [style=directed] (5) to (8);
		\draw [style=directed] (6) to (9.center);
	\end{pgfonlayer}
\end{tikzpicture}

%% file: figures/compl-fixed-arity-psi.tikz
\begin{tikzpicture}
	\begin{pgfonlayer}{nodelayer}
		\node [style=arbi] (0) at (0, 0) {$\psi$};
		\node [style=wire, label={left:$f$}] (1) at (0, 1.25) {};
		\node [style=wire, label={right:$g$}] (2) at (0.75, -1) {};
		\node [style=wire, label={left:$h$}] (3) at (-0.75, -1) {};
	\end{pgfonlayer}
	\begin{pgfonlayer}{edgelayer}
		\draw [style=directed] (0) to (1);
		\draw [style=directed] (2) to (0);
		\draw [style=directed] (3) to (0);
	\end{pgfonlayer}
\end{tikzpicture}

%% file: figures/psi-phi-contract.tikz
\begin{tikzpicture}
	\begin{pgfonlayer}{nodelayer}
		\node [style=none] (0) at (-1, -0.75) {$\psi$};
		\node [style=none] (1) at (-2.25, -0.25) {};
		\node [style=none] (2) at (0.25, -0.25) {};
		\node [style=none] (3) at (-1, -1.5) {};
		\node [style=wire, label={above:$f$}] (4) at (-1.75, 0.5) {};
		\node [style=none] (5) at (-1, 0) {};
		\node [style=none] (6) at (-0.25, 0) {};
		\node [style=wire, label={above:$d$}] (7) at (3.75, 0.5) {};
		\node [style=wire, label={above:$e$}] (8) at (4.5, 0.5) {};
		\node [style=none] (9) at (-1, -0.25) {};
		\node [style=none] (10) at (4.5, -0.25) {};
		\node [style=none] (11) at (-0.25, -0.25) {};
		\node [style=none] (12) at (-1.75, -0.25) {};
		\node [style=none] (13) at (3.75, -0.25) {};
		\node [style=none] (14) at (1, -0.25) {};
		\node [style=none] (15) at (3, -1.5) {};
		\node [style=none] (16) at (3, -0.75) {$\phi$};
		\node [style=none] (17) at (5, -0.25) {};
		\node [style=wire, label={above:$c$}] (18) at (3, 0.5) {};
		\node [style=none] (19) at (3, -0.25) {};
		\node [style=none] (20) at (1.5, 0) {};
		\node [style=none] (21) at (1.5, -0.25) {};
		\node [style=none] (22) at (2.25, 0) {};
		\node [style=none] (23) at (2.25, -0.25) {};
	\end{pgfonlayer}
	\begin{pgfonlayer}{edgelayer}
		\draw (1.center) to (3.center);
		\draw (3.center) to (2.center);
		\draw (2.center) to (1.center);
		\draw [style=directed] (12.center) to (4);
		\draw [style=directed, red] (5.center) to (9.center);
		\draw [style=directed, red] (6.center) to (11);
		\draw [style=directed] (7) to (13.center);
		\draw [style=directed] (8) to (10.center);
		\draw (14.center) to (15.center);
		\draw (15.center) to (17.center);
		\draw (17.center) to (14.center);
		\draw [style=directed] (18) to (19.center);
		\draw [style=undirected, red] (21.center) to (20.center);
		\draw [style=undirected, red] (23.center) to (22.center);
		\draw [red, style=undirected, in=90, out=90, looseness=1.00] (5.center) to (20.center);
		\draw [red, style=undirected, in=90, out=90, looseness=1.00] (22.center) to (6.center);
	\end{pgfonlayer}
\end{tikzpicture}

%% file: figures/compose-graphs.tikz
\begin{tikzpicture}
	\begin{pgfonlayer}{nodelayer}
		\node [style=arbi] (0) at (0, -0.75) {$\phi$};
		\node [style=wire, label={right:$c$}] (1) at (0.75, -1.5) {};
		\node [style=wire, label={[yshift=5pt]right:$d$}] (2) at (0, -2) {};
		\node [style=wire, label={left:$e$}] (3) at (-0.75, -1.5) {};
		\node [style=arbi] (4) at (0, 1) {$\psi$};
		\node [style=wire, label={[yshift=-5pt]right:$f$}] (5) at (0, 2) {};
	\end{pgfonlayer}
	\begin{pgfonlayer}{edgelayer}
		\draw [style=directed] (1) to (0);
		\draw [style=directed] (2) to (0);
		\draw [style=directed] (3) to (0);
		\draw [style=directed, red, in=-135, out=45, looseness=1.00] (0) to node[left, pos=0.7]{} (4);
		\draw [style=directed, red, in=-45, out=135, looseness=1.00] (0) to node[right, pos=0.7]{} (4);
		\draw [style=directed] (4) to (5);
	\end{pgfonlayer}
\end{tikzpicture}

%% file: figures/bb-example-noconn.tikz
\begin{tikzpicture}
	\begin{pgfonlayer}{nodelayer}
		\node [style=arbi] (0) at (1.5, 0.5) {$\psi$};
		\node [style=arbi] (1) at (-0.5, -0.5) {$\phi$};
		\node [style=bbox, label={$B$}] (2) at (0.75, 1.25) {};
		\node [style=none] (3) at (2.25, 1.25) {};
		\node [style=none] (4) at (2.25, -1) {};
		\node [style=none] (5) at (0.75, -1) {};
		\node [style=wire, label={{[yshift=-4pt]right:$a$}}] (6) at (-0.5, 0.75) {};
		\node [style=wire, label={{[yshift=4pt]right:$b$}}] (7) at (1.5, -0.75) {};
	\end{pgfonlayer}
	\begin{pgfonlayer}{edgelayer}
		\draw [style=boxedge] (2) to (5.center);
		\draw [style=boxedge] (5.center) to (4.center);
		\draw [style=boxedge] (4.center) to (3.center);
		\draw [style=boxedge] (3.center) to (2);
		\draw [style=directed] (1) to (6);
		\draw [style=directed] (7) to (0);
	\end{pgfonlayer}
\end{tikzpicture}

%% file: figures/bb-example.tikz
\begin{tikzpicture}
	\begin{pgfonlayer}{nodelayer}
		\node [style=arbi] (0) at (0, 1) {$\psi$};
		\node [style=arbi] (1) at (0, -1) {$\phi$};
		\node [style=bbox, label=$B$] (2) at (-0.75, 1.75) {};
		\node [style=none] (3) at (0.75, 1.75) {};
		\node [style=none] (4) at (0.75, 0.25) {};
		\node [style=none] (5) at (-0.75, 0.25) {};
	\end{pgfonlayer}
	\begin{pgfonlayer}{edgelayer}
		\draw [style=directed, arcout={{B}{0}{-60}{2mm}{1}}] (1) to (0);
		\draw [style=boxedge] (2) to (5.center);
		\draw [style=boxedge] (5.center) to (4.center);
		\draw [style=boxedge] (4.center) to (3.center);
		\draw [style=boxedge] (3.center) to (2);
	\end{pgfonlayer}
\end{tikzpicture}

%% file: figures/bb-example-cw.tikz
\begin{tikzpicture}
	\begin{pgfonlayer}{nodelayer}
		\node [style=arbi] (0) at (0, 1) {$\psi$};
		\node [style=arbi] (1) at (0, -1) {$\phi$};
		\node [style=bbox, label=$B$] (2) at (-0.75, 1.75) {};
		\node [style=none] (3) at (0.75, 1.75) {};
		\node [style=none] (4) at (0.75, 0.25) {};
		\node [style=none] (5) at (-0.75, 0.25) {};
	\end{pgfonlayer}
	\begin{pgfonlayer}{edgelayer}
		\draw [style=directed, arcout={{B}{0}{60}{2mm}{1}}] (1) to (0);
		\draw [style=boxedge] (2) to (5.center);
		\draw [style=boxedge] (5.center) to (4.center);
		\draw [style=boxedge] (4.center) to (3.center);
		\draw [style=boxedge] (3.center) to (2);
	\end{pgfonlayer}
\end{tikzpicture}

%% file: figures/nested-ex.tikz
\begin{tikzpicture}
	\begin{pgfonlayer}{nodelayer}
		\node [style=arbi] (0) at (0, 2) {$\phi$};
		\node [style=arbi] (1) at (0, -0.5) {$\phi$};
		\node [style=bbox, label={$B$}] (2) at (-0.75, 0.25) {};
		\node [style=bbox, label={$A$}] (3) at (-1.25, 0.75) {};
		\node [style=none] (4) at (0.75, 0.25) {};
		\node [style=none] (5) at (0.75, -1.75) {};
		\node [style=none] (6) at (-0.75, -1.75) {};
		\node [style=none] (7) at (-1.25, -2) {};
		\node [style=none] (8) at (1, -2) {};
		\node [style=none] (9) at (1, 0.75) {};
		\node [style=wire, label={[yshift=-4pt]right:$a$}] (10) at (0, 3) {};
		\node [style=wire, label={[yshift=4pt]right:$c$}] (11) at (0, -1.5) {};
	\end{pgfonlayer}
	\begin{pgfonlayer}{edgelayer}
		\draw [style=boxedge] (3) to (2);
		\draw [style=boxedge] (2) to (4.center);
		\draw [style=boxedge] (4.center) to (5.center);
		\draw [style=boxedge] (5.center) to (6.center);
		\draw [style=boxedge] (6.center) to (2);
		\draw [style=boxedge] (3) to (9.center);
		\draw [style=boxedge] (9.center) to (8.center);
		\draw [style=boxedge] (8.center) to (7.center);
		\draw [style=boxedge] (7.center) to (3);
		\draw [style=directed, arcin={{$B$}{0}{-65}{4mm}{1}}] (1) to (0);
		\draw [draw=none, arcin={{$A$}{0}{-65}{2mm}{1}}] (1) to (0);
		\draw [style=directed] (0) to (10);
		\draw [style=directed] (11) to (1);
	\end{pgfonlayer}
\end{tikzpicture}

%% file: figures/bb-ex1-kill.tikz
\begin{tikzpicture}
	\begin{pgfonlayer}{nodelayer}
		\node [style=arbi] (0) at (-0.5, 2.25) {$\xi$};
		\node [style=arbi] (1) at (0, -2.25) {$\zeta$};
		\node [style=wire, label={right:$e$}] (2) at (0, -3.5) {};
	\end{pgfonlayer}
	\begin{pgfonlayer}{edgelayer}
		\draw [style=directed] (2) to (1);
	\end{pgfonlayer}
\end{tikzpicture}

%% file: figures/bb-ex1.tikz
\begin{tikzpicture}
	\begin{pgfonlayer}{nodelayer}
		\node [style=arbi] (0) at (-0.5, 0.25) {$\phi$};
		\node [style=arbi] (1) at (0.75, -0.25) {$\psi$};
		\node [style=arbi] (2) at (-0.5, 2.25) {$\xi$};
		\node [style=arbi] (3) at (0, -2.25) {$\zeta$};
		\node [style=wire, label={right:$e$}] (4) at (0, -3.5) {};
		\node [style=bbox, label={$B$\ \ }] (5) at (-1.5, 1.25) {};
		\node [style=none] (6) at (1.75, -1) {};
		\node [style=none] (7) at (1.75, 1.25) {};
		\node [style=none] (8) at (-1.5, -1) {};
	\end{pgfonlayer}
	\begin{pgfonlayer}{edgelayer}
		\draw [style=directed, arcin={{}{0}{60}{1.5mm}{1}}] (0) to (2);
		\draw [style=directed, arcin={{}{15}{-90}{1.5mm}{1}}] (0) to (3);
		\draw [style=directed] (3) to (1);
		\draw [style=directed] (1) to (0);
		\draw [style=directed] (4) to (3);
		\draw [style=boxedge] (5) to (7.center);
		\draw [style=boxedge] (7.center) to (6.center);
		\draw [style=boxedge] (6.center) to (8.center);
		\draw [style=boxedge] (8.center) to (5);
	\end{pgfonlayer}
\end{tikzpicture}

%% file: figures/bb-ex1-exp.tikz
\begin{tikzpicture}
	\begin{pgfonlayer}{nodelayer}
		\node [style=arbi] (0) at (0.5, 0.25) {$\phi$};
		\node [style=arbi] (1) at (-0.75, 2.25) {$\xi$};
		\node [style=arbi] (2) at (1.75, -0.25) {$\psi$};
		\node [style=arbi] (3) at (-0.75, -2.25) {$\zeta$};
		\node [style=arbi] (4) at (-2.5, 0.25) {$\phi$};
		\node [style=arbi] (5) at (-1.25, -0.25) {$\psi$};
		\node [style=none] (6) at (2.5, -1) {};
		\node [style=bbox, label={$B$\ \ }] (7) at (-0.25, 1.25) {};
		\node [style=wire, label={right:$e$}] (8) at (-0.75, -3.5) {};
		\node [style=none] (9) at (2.5, 1.25) {};
		\node [style=none] (10) at (-0.25, -1) {};
	\end{pgfonlayer}
	\begin{pgfonlayer}{edgelayer}
		\draw [style=directed, arcin={{}{0}{45}{2mm}{1}}] (0) to (1);
		\draw [style=directed, arcin={{}{15}{-60}{3.5mm}{1}}] (0) to (3);
		\draw [style=directed] (3) to (2);
		\draw [style=directed] (2) to (0);
		\draw [style=directed] (5) to (4);
		\draw [style=directed] (4) to (1);
		\draw [style=directed] (4) to (3);
		\draw [style=directed] (3) to (5);
		\draw [style=directed] (8) to (3);
		\draw [style=boxedge] (7) to (9.center);
		\draw [style=boxedge] (9.center) to (6.center);
		\draw [style=boxedge] (6.center) to (10.center);
		\draw [style=boxedge] (10.center) to (7);
	\end{pgfonlayer}
\end{tikzpicture}

%% file: figures/n-ary-tree.tikz
\begin{tikzpicture}[dotpic]
	\begin{pgfonlayer}{nodelayer}
		\node [style=white] (0) at (3.5, 1) {};
		\node [style=white] (1) at (3, 0.5) {};
		\node [style=white] (2) at (1.5, -1) {};
		\node [style=none] (3) at (2, -0.5) {};
		\node [style=none] (4) at (2.5, 0) {};
		\node [style=none] (5) at (2.25, -0.25) {...};
		\node [style=none] (6) at (6.25, -1.75) {};
		\node [style=none] (7) at (5.25, -1.75) {};
		\node [style=none] (8) at (2.25, -1.75) {};
		\node [style=none] (9) at (0.75, -1.75) {};
		\node [style=white] (10) at (-1.75, 0) {};
		\node [style=none] (11) at (-2.5, -1) {};
		\node [style=none] (12) at (-1, -1) {};
		\node [style=none] (13) at (-1.75, -1) {...};
		\node [style=none] (14) at (0, 0) {$:=$};
		\node [style=none] (15) at (-1.75, 1) {};
		\node [style=none] (16) at (3.5, 2) {};
	\end{pgfonlayer}
	\begin{pgfonlayer}{edgelayer}
		\draw (2) to (3.center);
		\draw [style=directed] (4.center) to (1);
		\draw [style=directed] (1) to (0);
		\draw [style=directed] (9.center) to (2);
		\draw [style=directed] (8.center) to (2);
		\draw [style=directed] (7.center) to (1);
		\draw [style=directed] (6.center) to (0);
		\draw [style=directed] (11.center) to (10);
		\draw [style=directed] (12.center) to (10);
		\draw [style=directed] (10) to (15.center);
		\draw [style=directed] (0) to (16.center);
	\end{pgfonlayer}
\end{tikzpicture}

%% file: figures/n-ary-cotree.tikz
\begin{tikzpicture}[dotpic]
	\begin{pgfonlayer}{nodelayer}
		\node [style=gray] (0) at (3.5, -1) {};
		\node [style=gray] (1) at (3, -0.5) {};
		\node [style=gray] (2) at (1.5, 1) {};
		\node [style=none] (3) at (2, 0.5) {};
		\node [style=none] (4) at (2.5, 0) {};
		\node [style=none] (5) at (2.25, 0.25) {...};
		\node [style=none] (6) at (6.25, 1.75) {};
		\node [style=none] (7) at (5.25, 1.75) {};
		\node [style=none] (8) at (2.25, 1.75) {};
		\node [style=none] (9) at (0.75, 1.75) {};
		\node [style=gray] (10) at (-1.75, 0) {};
		\node [style=none] (11) at (-2.5, 1) {};
		\node [style=none] (12) at (-1, 1) {};
		\node [style=none] (13) at (-1.75, 1) {...};
		\node [style=none] (14) at (0, 0) {$:=$};
		\node [style=none] (15) at (-1.75, -1) {};
		\node [style=none] (16) at (3.5, -2) {};
	\end{pgfonlayer}
	\begin{pgfonlayer}{edgelayer}
		\draw (3.center) to (2);
		\draw [style=directed] (1) to (4.center);
		\draw [style=directed] (0) to (1);
		\draw [style=directed] (2) to (9.center);
		\draw [style=directed] (15.center) to (10);
		\draw [style=directed] (16.center) to (0);
		\draw [style=directed] (10) to (11.center);
		\draw [style=directed] (10) to (12.center);
		\draw [style=directed] (2) to (8.center);
		\draw [style=directed] (1) to (7.center);
		\draw [style=directed] (0) to (6.center);
	\end{pgfonlayer}
\end{tikzpicture}

%% file: figures/unit-left.tikz
\begin{tikzpicture}[dotpic]
	\begin{pgfonlayer}{nodelayer}
		\node [style=white] (0) at (-2, 0) {};
		\node [style=wire] (1) at (-2, 1) {};
		\node [style=wire] (2) at (-1.5, -1) {};
		\node [style=white] (3) at (-2.5, -0.75) {};
		\node [style=none] (4) at (-0.75, 0) {$=$};
		\node [style=wire] (5) at (0, -1) {};
		\node [style=wire] (6) at (0, 1) {};
	\end{pgfonlayer}
	\begin{pgfonlayer}{edgelayer}
		\draw [style=directed] (0) to (1);
		\draw [style=directed] (3) to (0);
		\draw [style=directed] (2) to (0);
		\draw [style=directed] (5) to (6);
	\end{pgfonlayer}
\end{tikzpicture}

%% file: figures/unit-right.tikz
\begin{tikzpicture}[dotpic]
	\begin{pgfonlayer}{nodelayer}
		\node [style=white] (0) at (-2, 0) {};
		\node [style=wire] (1) at (-2, 1) {};
		\node [style=wire] (2) at (-2.5, -1) {};
		\node [style=white] (3) at (-1.5, -0.75) {};
		\node [style=none] (4) at (-0.75, 0) {$=$};
		\node [style=wire] (5) at (0, -1) {};
		\node [style=wire] (6) at (0, 1) {};
	\end{pgfonlayer}
	\begin{pgfonlayer}{edgelayer}
		\draw [style=directed] (0) to (1);
		\draw [style=directed] (3) to (0);
		\draw [style=directed] (2) to (0);
		\draw [style=directed] (5) to (6);
	\end{pgfonlayer}
\end{tikzpicture}

%% file: figures/assoc.tikz
\begin{tikzpicture}[dotpic]
	\begin{pgfonlayer}{nodelayer}
		\node [style=white] (0) at (-2, -0.25) {};
		\node [style=white] (1) at (-1.5, 0.5) {};
		\node [style=wire] (2) at (-2.5, -1) {};
		\node [style=wire] (3) at (-1.5, -1) {};
		\node [style=wire] (4) at (-0.5, -1) {};
		\node [style=wire] (5) at (-1.5, 1.25) {};
		\node [style=none] (6) at (0, 0) {$=$};
		\node [style=wire] (7) at (1.5, -1) {};
		\node [style=white] (8) at (2, -0.25) {};
		\node [style=wire] (9) at (0.5, -1) {};
		\node [style=white] (10) at (1.5, 0.5) {};
		\node [style=wire] (11) at (1.5, 1.25) {};
		\node [style=wire] (12) at (2.5, -1) {};
	\end{pgfonlayer}
	\begin{pgfonlayer}{edgelayer}
		\draw [style=directed] (1) to (5);
		\draw [style=directed] (0) to (1);
		\draw [style=directed] (3) to (0);
		\draw [style=directed] (2) to (0);
		\draw [style=directed] (4) to (1);
		\draw [style=directed] (10) to (11);
		\draw [style=directed] (8) to (10);
		\draw [style=directed] (7) to (8);
		\draw [style=directed] (12) to (8);
		\draw [style=directed] (9) to (10);
	\end{pgfonlayer}
\end{tikzpicture}

%% file: figures/n-mult-rec.tikz
\begin{tikzpicture}[dotpic]
	\begin{pgfonlayer}{nodelayer}
		\node [style=wire] (0) at (-2.5, -1.25) {};
		\node [style=none] (1) at (-3, -1.75) {};
		\node [style=none] (2) at (-2, -0.75) {};
		\node [style=wire] (3) at (-1.75, 1.75) {};
		\node [style=bbox, label={$A$}] (4) at (-3, -0.75) {};
		\node [style=none] (5) at (-2, -1.75) {};
		\node [style=white dot] (6) at (-1.75, 0.5) {};
		\node [style=none] (7) at (0, 0) {$=$};
		\node [style=wire] (8) at (-1, -1.25) {};
		\node [style=none] (9) at (2.25, -0.75) {};
		\node [style=wire] (10) at (2.5, 1.75) {};
		\node [style=white dot] (11) at (2.5, 1) {};
		\node [style=none] (12) at (2.25, -1.75) {};
		\node [style=bbox, label={$A$}] (13) at (1.25, -0.75) {};
		\node [style=wire] (14) at (1.75, -1.25) {};
		\node [style=none] (15) at (1.25, -1.75) {};
		\node [style=wire] (16) at (3.25, -1.25) {};
		\node [style=white dot] (17) at (1.75, 0.25) {};
	\end{pgfonlayer}
	\begin{pgfonlayer}{edgelayer}
		\draw [style=directed] (6) to (3);
		\draw [style=right arcin] (0) to (6);
		\draw [style=boxedge] (2.center) to (4);
		\draw [style=boxedge] (5.center) to (2.center);
		\draw [style=boxedge] (1.center) to (5.center);
		\draw [style=boxedge] (4) to (1.center);
		\draw [style=directed] (8) to (6);
		\draw [style=directed] (11) to (10);
		\draw [style=boxedge] (9.center) to (13);
		\draw [style=boxedge] (12.center) to (9.center);
		\draw [style=boxedge] (15.center) to (12.center);
		\draw [style=boxedge] (13) to (15.center);
		\draw [style=directed] (16) to (11);
		\draw [style=directed] (17) to (11);
		\draw [style=right arcin] (14) to (17);
	\end{pgfonlayer}
\end{tikzpicture}

%% file: figures/counit-left.tikz
\begin{tikzpicture}[dotpic]
	\begin{pgfonlayer}{nodelayer}
		\node [style=gray] (0) at (-1.25, 0) {};
		\node [style=wire] (1) at (-1.25, -1) {};
		\node [style=wire] (2) at (-0.75, 1) {};
		\node [style=gray] (3) at (-1.75, 1) {};
		\node [style=none] (4) at (0, 0) {$=$};
		\node [style=wire] (5) at (0.75, 1) {};
		\node [style=wire] (6) at (0.75, -1) {};
	\end{pgfonlayer}
	\begin{pgfonlayer}{edgelayer}
		\draw [style=directed] (1) to (0);
		\draw [style=directed] (0) to (3);
		\draw [style=directed] (0) to (2);
		\draw [style=directed] (6) to (5);
	\end{pgfonlayer}
\end{tikzpicture}

%% file: figures/counit-right.tikz
\begin{tikzpicture}[dotpic]
	\begin{pgfonlayer}{nodelayer}
		\node [style=gray] (0) at (-1.25, 0) {};
		\node [style=wire] (1) at (-1.25, -1) {};
		\node [style=wire] (2) at (-1.75, 1) {};
		\node [style=gray] (3) at (-0.75, 1) {};
		\node [style=none] (4) at (0, 0) {$=$};
		\node [style=wire] (5) at (0.75, 1) {};
		\node [style=wire] (6) at (0.75, -1) {};
	\end{pgfonlayer}
	\begin{pgfonlayer}{edgelayer}
		\draw [style=directed] (1) to (0);
		\draw [style=directed] (0) to (3);
		\draw [style=directed] (0) to (2);
		\draw [style=directed] (6) to (5);
	\end{pgfonlayer}
\end{tikzpicture}

%% file: figures/coassoc.tikz
\begin{tikzpicture}[dotpic]
	\begin{pgfonlayer}{nodelayer}
		\node [style=gray] (0) at (-2, 0.5) {};
		\node [style=gray] (1) at (-1.5, -0.25) {};
		\node [style=none] (2) at (-2.5, 1.25) {};
		\node [style=none] (3) at (-1.5, 1.25) {};
		\node [style=none] (4) at (-0.5, 1.25) {};
		\node [style=wire] (5) at (-1.5, -1) {};
		\node [style=none] (6) at (0, 0.25) {$=$};
		\node [style=none] (7) at (1.5, 1.25) {};
		\node [style=gray] (8) at (2, 0.5) {};
		\node [style=none] (9) at (0.5, 1.25) {};
		\node [style=gray] (10) at (1.5, -0.25) {};
		\node [style=wire] (11) at (1.5, -1) {};
		\node [style=none] (12) at (2.5, 1.25) {};
	\end{pgfonlayer}
	\begin{pgfonlayer}{edgelayer}
		\draw [style=directed] (5) to (1);
		\draw [style=directed] (1) to (0);
		\draw [style=directed] (0) to (3.center);
		\draw [style=directed] (0) to (2.center);
		\draw [style=directed] (1) to (4.center);
		\draw [style=directed] (11) to (10);
		\draw [style=directed] (10) to (8);
		\draw [style=directed] (8) to (7.center);
		\draw [style=directed] (8) to (12.center);
		\draw [style=directed] (10) to (9.center);
	\end{pgfonlayer}
\end{tikzpicture}

%% file: figures/bialg1.tikz
\begin{tikzpicture}[dotpic]
	\begin{pgfonlayer}{nodelayer}
		\node [style=white dot] (0) at (1.5, -0.5) {};
		\node [style=white dot] (1) at (-2.75, 0.75) {};
		\node [style=wire] (2) at (-1.25, -1.5) {};
		\node [style=gray dot] (3) at (1.5, 0.5) {};
		\node [style=wire] (4) at (-2.75, -1.5) {};
		\node [style=gray dot] (5) at (-2.75, -0.75) {};
		\node [style=wire] (6) at (2.25, 1.5) {};
		\node [style=wire] (7) at (-1.25, 1.5) {};
		\node [style=wire] (8) at (-2.75, 1.5) {};
		\node [style=wire] (9) at (0.75, 1.5) {};
		\node [style=none] (10) at (0, 0) {$=$};
		\node [style=wire] (11) at (-1.25, 1.5) {};
		\node [style=wire] (12) at (2.25, -1.5) {};
		\node [style=wire] (13) at (0.75, -1.5) {};
		\node [style=gray dot] (14) at (-1.25, -0.75) {};
		\node [style=wire] (15) at (2.25, 1.5) {};
		\node [style=white dot] (16) at (-1.25, 0.75) {};
	\end{pgfonlayer}
	\begin{pgfonlayer}{edgelayer}
		\draw [style=directed] (2) to (14);
		\draw [style=directed] (16) to (11);
		\draw [style=directed] (4) to (5);
		\draw [style=directed] (1) to (8);
		\draw [style=directed] (5) to (1);
		\draw [style=directed] (5) to (16);
		\draw [style=directed] (14) to (16);
		\draw [style=directed] (14) to (1);
		\draw [style=directed] (0) to (3);
		\draw [style=directed] (12) to (0);
		\draw [style=directed] (13) to (0);
		\draw [style=directed] (3) to (6);
		\draw [style=directed] (3) to (9);
	\end{pgfonlayer}
\end{tikzpicture}

%% file: figures/bialg2.tikz
\begin{tikzpicture}[dotpic]
	\begin{pgfonlayer}{nodelayer}
		\node [style=wire] (0) at (-2.25, 1) {};
		\node [style=white dot] (1) at (-1.5, -1) {};
		\node [style=gray dot] (2) at (-1.5, 0) {};
		\node [style=none] (3) at (0, 0) {$=$};
		\node [style=wire] (4) at (-0.75, 1) {};
		\node [style=wire] (5) at (1.25, 0.75) {};
		\node [style=white dot] (6) at (1.25, -0.25) {};
		\node [style=white dot] (7) at (2.25, -0.25) {};
		\node [style=wire] (8) at (2.25, 0.75) {};
	\end{pgfonlayer}
	\begin{pgfonlayer}{edgelayer}
		\draw [style=directed] (1) to (2);
		\draw [style=directed] (2) to (0);
		\draw [style=directed] (2) to (4);
		\draw [style=directed] (6) to (5);
		\draw [style=directed] (7) to (8);
	\end{pgfonlayer}
\end{tikzpicture}

%% file: figures/bialg3.tikz
\begin{tikzpicture}[dotpic]
	\begin{pgfonlayer}{nodelayer}
		\node [style=wire] (0) at (-2.25, -1) {};
		\node [style=gray dot] (1) at (-1.5, 1) {};
		\node [style=white dot] (2) at (-1.5, 0) {};
		\node [style=none] (3) at (0, 0) {$=$};
		\node [style=wire] (4) at (-0.75, -1) {};
		\node [style=wire] (5) at (1.25, -0.5) {};
		\node [style=gray dot] (6) at (1.25, 0.5) {};
		\node [style=gray dot] (7) at (2.25, 0.5) {};
		\node [style=wire] (8) at (2.25, -0.5) {};
	\end{pgfonlayer}
	\begin{pgfonlayer}{edgelayer}
		\draw [style=directed] (2) to (1);
		\draw [style=directed] (0) to (2);
		\draw [style=directed] (4) to (2);
		\draw [style=directed] (5) to (6);
		\draw [style=directed] (8) to (7);
	\end{pgfonlayer}
\end{tikzpicture}

%% file: figures/bialg4.tikz
\begin{tikzpicture}[dotpic]
	\begin{pgfonlayer}{nodelayer}
		\node [style=gray dot] (0) at (-1, 0.5) {};
		\node [style=none] (1) at (0, 0) {$=$};
		\node [style=empty diagram] (2) at (2, 0) {};
		\node [style=white dot] (3) at (-1, -0.5) {};
	\end{pgfonlayer}
	\begin{pgfonlayer}{edgelayer}
		\draw [style=directed] (3) to (0);
	\end{pgfonlayer}
\end{tikzpicture}

%% file: figures/gen-bialg-bbox.tikz
\begin{tikzpicture}[dotpic]
	\begin{pgfonlayer}{nodelayer}
		\node [style=wire] (0) at (-1.5, -2) {};
		\node [style=white] (1) at (-1.5, -0.5) {};
		\node [style=wire] (2) at (-1.5, 2) {};
		\node [style=gray] (3) at (-1.5, 0.5) {};
		\node [style=none] (4) at (0, 0) {$=$};
		\node [style=gray] (5) at (1.75, -1.5) {};
		\node [style=wire] (6) at (1.75, -2.25) {};
		\node [style=white] (7) at (1.75, 1.5) {};
		\node [style=wire] (8) at (1.75, 2.25) {};
		\node [style=none] (9) at (-1, 2.5) {};
		\node [style=bbox] (10) at (-2, 2.5) {};
		\node [style=none] (11) at (-1, 1.5) {};
		\node [style=none] (12) at (-2, 1.5) {};
		\node [style=none] (13) at (-1, -1.5) {};
		\node [style=bbox] (14) at (-2, -1.5) {};
		\node [style=none] (15) at (-1, -2.5) {};
		\node [style=none] (16) at (-2, -2.5) {};
		\node [style=none] (17) at (2.5, -0.5) {};
		\node [style=bbox] (18) at (1, -0.5) {};
		\node [style=none] (19) at (2.5, -2.5) {};
		\node [style=none] (20) at (1, -2.5) {};
		\node [style=none] (21) at (2.5, 2.5) {};
		\node [style=bbox] (22) at (1, 2.5) {};
		\node [style=none] (23) at (2.5, 0.5) {};
		\node [style=none] (24) at (1, 0.5) {};
		\node [style=none] (25) at (1.75, 0) {};
	\end{pgfonlayer}
	\begin{pgfonlayer}{edgelayer}
		\draw [style=right arcin] (0) to (1);
		\draw [style=right arcout] (3) to (2);
		\draw [style=directed] (1) to (3);
		\draw [style=directed] (6) to (5);
		\draw [style=directed] (7) to (8);
		\draw [style=boxedge] (10) to (9.center);
		\draw [style=boxedge] (9.center) to (11.center);
		\draw [style=boxedge] (11.center) to (12.center);
		\draw [style=boxedge] (12.center) to (10);
		\draw [style=boxedge] (14) to (13.center);
		\draw [style=boxedge] (13.center) to (15.center);
		\draw [style=boxedge] (15.center) to (16.center);
		\draw [style=boxedge] (16.center) to (14);
		\draw [style=boxedge] (18) to (17.center);
		\draw [style=boxedge] (17.center) to (19.center);
		\draw [style=boxedge] (19.center) to (20.center);
		\draw [style=boxedge] (20.center) to (18);
		\draw [style=boxedge] (22) to (21.center);
		\draw [style=boxedge] (21.center) to (23.center);
		\draw [style=boxedge] (23.center) to (24.center);
		\draw [style=boxedge] (24.center) to (22);
		\draw [style=right arcin] (25.center) to (7);
		\draw [style=uright arcout] (5) to (25.center);
	\end{pgfonlayer}
\end{tikzpicture}

%% file: figures/step_case_nolab.tikz
\begin{tikzpicture}[dotpic]
	\begin{pgfonlayer}{nodelayer}
		\node [style=white] (0) at (-7.75, -1.25) {};
		\node [style=gray] (1) at (-7.75, -0.25) {};
		\node [style=none] (2) at (-6, 0) {$=$};
		\node [style=wire] (3) at (-8.5, 1) {};
		\node [style=wire] (4) at (6.5, 0.75) {};
		\node [style=white] (5) at (6.5, -0.25) {};
		\node [style=wire] (6) at (-7, 1) {};
		\node [style=white] (7) at (7.75, -0.25) {};
		\node [style=wire] (8) at (7.75, 0.75) {};
		\node [style=bbox] (9) at (6, 1) {};
		\node [style=none] (10) at (-9, 0.5) {};
		\node [style=none] (11) at (-8, 0.5) {};
		\node [style=none] (12) at (-8, 1.5) {};
		\node [style=none] (13) at (6, -0.75) {};
		\node [style=none] (14) at (7, -0.75) {};
		\node [style=none] (15) at (7, 1) {};
		\node [style=bbox] (16) at (-9, 1.5) {};
		\node [style=none] (17) at (-6, 0.75) {};
		\node [style=wire] (18) at (-2.75, 0.5) {};
		\node [style=gray] (19) at (-3.25, -0.5) {};
		\node [style=none] (20) at (-4.5, 0.75) {};
		\node [style=none] (21) at (-3.5, 1.75) {};
		\node [style=none] (22) at (-3.5, 0.75) {};
		\node [style=white] (23) at (-3.25, -1.5) {};
		\node [style=wire] (24) at (-4, 1.25) {};
		\node [style=bbox] (25) at (-4.5, 1.75) {};
		\node [style=gray] (26) at (-4, 0) {};
		\node [style=wire] (27) at (1, 1.25) {};
		\node [style=none] (28) at (1.5, 0.75) {};
		\node [style=gray] (29) at (1, 0) {};
		\node [style=none] (30) at (1.5, 1.75) {};
		\node [style=none] (31) at (0.5, 0.75) {};
		\node [style=white] (32) at (1, -1) {};
		\node [style=none] (33) at (-1.25, 0.75) {};
		\node [style=bbox] (34) at (0.5, 1.75) {};
		\node [style=none] (35) at (-1.25, 0) {$=$};
		\node [style=white] (36) at (2.25, 0) {};
		\node [style=wire] (37) at (2.25, 1) {};
		\node [style=none] (38) at (4, 0.75) {\footnotesize\textit{i.h.}};
		\node [style=none] (39) at (4, 0) {$=$};
	\end{pgfonlayer}
	\begin{pgfonlayer}{edgelayer}
		\draw [style=directed] (0) to (1);
		\draw [style=right arcout] (1) to (3);
		\draw [style=directed] (5) to (4);
		\draw [style=directed] (1) to (6);
		\draw [style=directed] (7) to (8);
		\draw [style=boxedge] (12.center) to (11.center);
		\draw [style=boxedge] (14.center) to (13.center);
		\draw [style=boxedge] (16) to (12.center);
		\draw [style=boxedge] (15.center) to (14.center);
		\draw [style=boxedge] (13.center) to (9);
		\draw [style=boxedge] (11.center) to (10.center);
		\draw [style=boxedge] (10.center) to (16);
		\draw [style=boxedge] (9) to (15.center);
		\draw [style=directed] (23) to (19);
		\draw [style=directed] (19) to (18);
		\draw [style=boxedge] (21.center) to (22.center);
		\draw [style=boxedge] (25) to (21.center);
		\draw [style=boxedge] (22.center) to (20.center);
		\draw [style=boxedge] (20.center) to (25);
		\draw [style=directed] (19) to (26);
		\draw [style=right arcout] (26) to (24);
		\draw [style=boxedge] (30.center) to (28.center);
		\draw [style=boxedge] (34) to (30.center);
		\draw [style=boxedge] (28.center) to (31.center);
		\draw [style=boxedge] (31.center) to (34);
		\draw [style=right arcout] (29) to (27);
		\draw [style=directed] (36) to (37);
		\draw [style=directed] (32) to (29);
	\end{pgfonlayer}
\end{tikzpicture}

%% file: figures/bialg_lemma_step.tikz
\begin{tikzpicture}[dotpic]
	\begin{pgfonlayer}{nodelayer}
		\node [style=none] (0) at (14.25, 1.75) {};
		\node [style=bbox] (1) at (13.25, 1.75) {};
		\node [style=none] (2) at (14.25, 0.25) {};
		\node [style=none] (3) at (13.25, 0.25) {};
		\node [style=none] (4) at (-1.5, 1.25) {};
		\node [style=none] (5) at (-1.5, 2.25) {};
		\node [style=none] (6) at (-2.5, 1.25) {};
		\node [style=bbox] (7) at (-2.5, 2.25) {};
		\node [style=gray] (8) at (15.25, -1) {};
		\node [style=gray] (9) at (13.75, -1) {};
		\node [style=gray] (10) at (-2, 0.5) {};
		\node [style=white] (11) at (13.75, 0.75) {};
		\node [style=white] (12) at (-2, -0.5) {};
		\node [style=wire] (13) at (-2, 1.75) {};
		\node [style=wire] (14) at (13.75, 1.5) {};
		\node [style=wire] (15) at (15.25, -2) {};
		\node [style=wire] (16) at (13.75, -2) {};
		\node [style=none] (17) at (12, 0) {$=$};
		\node [style=wire] (18) at (15.25, 1.5) {};
		\node [style=white] (19) at (15.25, 0.75) {};
		\node [style=gray] (20) at (3, -1.5) {};
		\node [style=none] (21) at (3.5, 2.25) {};
		\node [style=bbox] (22) at (2.5, 2.25) {};
		\node [style=white] (23) at (5, 0) {};
		\node [style=white] (24) at (3, 1) {};
		\node [style=gray] (25) at (5, -1.5) {};
		\node [style=none] (26) at (3.5, 0.5) {};
		\node [style=wire] (27) at (3, 1.75) {};
		\node [style=wire] (28) at (5, -2.25) {};
		\node [style=wire] (29) at (5, 1.25) {};
		\node [style=wire] (30) at (3, -2.25) {};
		\node [style=none] (31) at (2.5, 0.5) {};
		\node [style=gray] (32) at (3, -0.5) {};
		\node [style=gray] (33) at (4, -0.25) {};
		\node [style=white] (34) at (-0.5, -0.5) {};
		\node [style=gray] (35) at (-0.5, -1.75) {};
		\node [style=gray] (36) at (-2, -1.75) {};
		\node [style=wire] (37) at (-0.5, -2.5) {};
		\node [style=wire] (38) at (-2, -2.5) {};
		\node [style=none] (39) at (1, 0) {$=$};
		\node [style=wire] (40) at (-0.5, 1.25) {};
		\node [style=none] (41) at (1, 0.75) {\footnotesize\textit{i.h.}};
		\node [style=gray] (42) at (-5.25, -0.5) {};
		\node [style=none] (43) at (-6.5, 1) {};
		\node [style=gray] (44) at (-6, 0.25) {};
		\node [style=white] (45) at (-5.25, -1.5) {};
		\node [style=none] (46) at (-3.5, 0) {$=$};
		\node [style=wire] (47) at (-4.75, 1) {};
		\node [style=none] (48) at (-5.5, 2) {};
		\node [style=wire] (49) at (-4.75, -2.25) {};
		\node [style=none] (50) at (-5.5, 1) {};
		\node [style=wire] (51) at (-5.75, -2.25) {};
		\node [style=bbox] (52) at (-6.5, 2) {};
		\node [style=wire] (53) at (-6, 1.5) {};
		\node [style=gray] (54) at (-9.5, -0.25) {};
		\node [style=white] (55) at (-9.5, -1.25) {};
		\node [style=wire] (56) at (-9, -2) {};
		\node [style=wire] (57) at (-10, -2) {};
		\node [style=none] (58) at (-9.5, 1.75) {};
		\node [style=bbox] (59) at (-10.5, 1.75) {};
		\node [style=wire] (60) at (-9, 1.25) {};
		\node [style=none] (61) at (-9.5, 0.75) {};
		\node [style=wire] (62) at (-10, 1.25) {};
		\node [style=none] (63) at (-7.75, 0) {$=$};
		\node [style=none] (64) at (-10.5, 0.75) {};
		\node [style=white] (65) at (8.5, 1) {};
		\node [style=wire] (66) at (8.5, -2.25) {};
		\node [style=gray] (67) at (10.5, -1.5) {};
		\node [style=none] (68) at (9, 0.5) {};
		\node [style=wire] (69) at (8.5, 1.75) {};
		\node [style=none] (70) at (8, 0.5) {};
		\node [style=bbox] (71) at (8, 2.25) {};
		\node [style=gray] (72) at (8.5, -1.5) {};
		\node [style=none] (73) at (9, 2.25) {};
		\node [style=wire] (74) at (10.5, 1.25) {};
		\node [style=white] (75) at (10.5, 0) {};
		\node [style=wire] (76) at (10.5, -2.25) {};
		\node [style=gray] (77) at (8.5, -0.5) {};
		\node [style=none] (78) at (6.5, 0) {$=$};
	\end{pgfonlayer}
	\begin{pgfonlayer}{edgelayer}
		\draw [style=boxedge] (0.center) to (1);
		\draw [style=boxedge] (1) to (3.center);
		\draw [style=boxedge] (3.center) to (2.center);
		\draw [style=boxedge] (2.center) to (0.center);
		\draw [style=boxedge] (5.center) to (7);
		\draw [style=boxedge] (7) to (6.center);
		\draw [style=boxedge] (6.center) to (4.center);
		\draw [style=boxedge] (4.center) to (5.center);
		\draw [style=directed] (15) to (8);
		\draw [style=directed] (16) to (9);
		\draw [style=right arcout] (9) to (11);
		\draw [style=right arcout] (8) to (11);
		\draw [style=directed] (11) to (14);
		\draw [style=directed] (12) to (10);
		\draw [style=right arcout] (10) to (13);
		\draw [style=directed] (19) to (18);
		\draw [style=directed] (9) to (19);
		\draw [style=directed] (8) to (19);
		\draw [style=boxedge] (21.center) to (22);
		\draw [style=boxedge] (22) to (31.center);
		\draw [style=boxedge] (31.center) to (26.center);
		\draw [style=boxedge] (26.center) to (21.center);
		\draw [style=directed] (30) to (20);
		\draw [style=directed] (28) to (25);
		\draw [style=directed] (24) to (27);
		\draw [style=directed] (23) to (29);
		\draw [style=directed] (25) to (23);
		\draw [style=directed] (20) to (23);
		\draw [style=directed] (20) to (32);
		\draw [style=right arcout] (32) to (24);
		\draw [style=directed] (25) to (33);
		\draw [style=right arcout] (33) to (24);
		\draw [style=directed] (37) to (35);
		\draw [style=directed] (38) to (36);
		\draw [style=directed] (34) to (40);
		\draw [style=directed] (36) to (34);
		\draw [style=directed] (35) to (34);
		\draw [style=directed] (36) to (12);
		\draw [style=directed] (35) to (12);
		\draw [style=boxedge] (48.center) to (52);
		\draw [style=boxedge] (52) to (43.center);
		\draw [style=boxedge] (43.center) to (50.center);
		\draw [style=boxedge] (50.center) to (48.center);
		\draw [style=right arcout] (44) to (53);
		\draw [style=directed] (45) to (42);
		\draw [style=directed] (51) to (45);
		\draw [style=directed] (49) to (45);
		\draw [style=directed] (42) to (47);
		\draw [style=directed] (42) to (44);
		\draw [style=boxedge] (58.center) to (59);
		\draw [style=boxedge] (59) to (64.center);
		\draw [style=boxedge] (64.center) to (61.center);
		\draw [style=boxedge] (61.center) to (58.center);
		\draw [style=directed] (55) to (54);
		\draw [style=directed] (57) to (55);
		\draw [style=directed] (56) to (55);
		\draw [style=directed] (54) to (60);
		\draw [style=right arcout] (54) to (62);
		\draw [style=boxedge] (73.center) to (71);
		\draw [style=boxedge] (71) to (70.center);
		\draw [style=boxedge] (70.center) to (68.center);
		\draw [style=boxedge] (68.center) to (73.center);
		\draw [style=directed] (66) to (72);
		\draw [style=directed] (76) to (67);
		\draw [style=directed] (65) to (69);
		\draw [style=directed] (75) to (74);
		\draw [style=directed] (67) to (75);
		\draw [style=directed] (72) to (75);
		\draw [style=directed] (72) to (77);
		\draw [style=right arcout] (77) to (65);
		\draw [style=right arcout] (67) to (65);
	\end{pgfonlayer}
\end{tikzpicture}

%% file: figures/bialg_bbox_step.tikz
\begin{tikzpicture}[dotpic]
	\begin{pgfonlayer}{nodelayer}
		\node [style=bbox] (0) at (3.75, 2.75) {};
		\node [style=none] (1) at (6, 2.75) {};
		\node [style=none] (2) at (3.75, -0.25) {};
		\node [style=none] (3) at (6, -0.25) {};
		\node [style=none] (4) at (-10, 1.25) {};
		\node [style=bbox] (5) at (-10, 2.25) {};
		\node [style=none] (6) at (-9, 1.25) {};
		\node [style=none] (7) at (-9, 2.25) {};
		\node [style=gray] (8) at (4.5, -1.75) {};
		\node [style=gray] (9) at (-9.5, 0.5) {};
		\node [style=white] (10) at (5.25, 1.5) {};
		\node [style=white] (11) at (-9.5, -0.5) {};
		\node [style=wire] (12) at (-9.5, 1.75) {};
		\node [style=wire] (13) at (-8.75, -2.25) {};
		\node [style=wire] (14) at (5.25, 2.25) {};
		\node [style=wire] (15) at (4.5, -2.5) {};
		\node [style=none] (16) at (-7.75, 0) {$=$};
		\node [style=none] (17) at (5, -2.75) {};
		\node [style=none] (18) at (5, -0.75) {};
		\node [style=bbox] (19) at (4, -0.75) {};
		\node [style=none] (20) at (4, -2.75) {};
		\node [style=none] (21) at (-9.75, -2.75) {};
		\node [style=bbox] (22) at (-10.75, -1.75) {};
		\node [style=none] (23) at (-9.75, -1.75) {};
		\node [style=none] (24) at (-10.75, -2.75) {};
		\node [style=gray] (25) at (5.75, -1.75) {};
		\node [style=wire] (26) at (5.75, -2.5) {};
		\node [style=wire] (27) at (-10.25, -2.25) {};
		\node [style=none] (28) at (-5.75, -3) {};
		\node [style=none] (29) at (-5, 2.5) {};
		\node [style=white] (30) at (-5.5, -0.5) {};
		\node [style=wire] (31) at (-5.5, 2) {};
		\node [style=none] (32) at (-6, 1.5) {};
		\node [style=bbox] (33) at (-6, 2.5) {};
		\node [style=none] (34) at (-5.75, -2) {};
		\node [style=wire] (35) at (-5, -2.25) {};
		\node [style=wire] (36) at (-6.25, -2.5) {};
		\node [style=bbox] (37) at (-6.75, -2) {};
		\node [style=none] (38) at (-6.75, -3) {};
		\node [style=gray] (39) at (-5.5, 0.75) {};
		\node [style=none] (40) at (-5, 1.5) {};
		\node [style=white] (41) at (-6.25, -1) {};
		\node [style=none] (42) at (-3.25, 0) {$=$};
		\node [style=none] (43) at (-0.25, -3) {};
		\node [style=wire] (44) at (-0.75, -2.5) {};
		\node [style=wire] (45) at (0.5, -2.25) {};
		\node [style=bbox] (46) at (-1.25, -2) {};
		\node [style=none] (47) at (-1.25, -3) {};
		\node [style=white] (48) at (-0.75, -1) {};
		\node [style=none] (49) at (-0.25, -2) {};
		\node [style=white] (50) at (-0.75, 1.25) {};
		\node [style=none] (51) at (0, 2.5) {};
		\node [style=wire] (52) at (-0.75, 2) {};
		\node [style=none] (53) at (-1.5, 0.75) {};
		\node [style=bbox] (54) at (-1.5, 2.5) {};
		\node [style=none] (55) at (0, 0.75) {};
		\node [style=gray] (56) at (-0.75, 0) {};
		\node [style=gray] (57) at (0.5, 0) {};
		\node [style=none] (58) at (2.25, 0) {$=$};
		\node [style=none] (59) at (2.25, 1) {\footnotesize\textit{i.h.}};
		\node [style=white] (60) at (4.5, 0.75) {};
		\node [style=none] (61) at (7.5, 0) {$=$};
		\node [style=none] (62) at (-3.25, 1) {\footnotesize(\ref{lem_bialg})};
		\node [style=none] (63) at (4.5, -0.5) {};
		\node [style=none] (64) at (8.75, -2.5) {};
		\node [style=none] (65) at (11.25, 2.5) {};
		\node [style=white] (66) at (10.5, 1.25) {};
		\node [style=gray] (67) at (11.25, -1.25) {};
		\node [style=gray] (68) at (9.5, -1.25) {};
		\node [style=none] (69) at (11.25, 0.25) {};
		\node [style=none] (70) at (10.25, -2.5) {};
		\node [style=bbox] (71) at (9.75, 2.5) {};
		\node [style=bbox] (72) at (8.75, -0.25) {};
		\node [style=wire] (73) at (9.5, -2.25) {};
		\node [style=wire] (74) at (11.25, -2.25) {};
		\node [style=none] (75) at (10.25, -0.25) {};
		\node [style=wire] (76) at (10.5, 2.25) {};
		\node [style=none] (77) at (9.75, 0.25) {};
		\node [style=none] (78) at (10, 0) {};
	\end{pgfonlayer}
	\begin{pgfonlayer}{edgelayer}
		\draw [style=boxedge] (0) to (1.center);
		\draw [style=boxedge] (1.center) to (3.center);
		\draw [style=boxedge] (3.center) to (2.center);
		\draw [style=boxedge] (2.center) to (0);
		\draw [style=boxedge] (5) to (7.center);
		\draw [style=boxedge] (7.center) to (6.center);
		\draw [style=boxedge] (6.center) to (4.center);
		\draw [style=boxedge] (4.center) to (5);
		\draw [style=directed] (15) to (8);
		\draw [style=directed] (10) to (14);
		\draw [style=directed] (11) to (9);
		\draw [style=right arcout] (9) to (12);
		\draw [style=boxedge] (19) to (18.center);
		\draw [style=boxedge] (18.center) to (17.center);
		\draw [style=boxedge] (17.center) to (20.center);
		\draw [style=boxedge] (20.center) to (19);
		\draw [style=boxedge] (22) to (23.center);
		\draw [style=boxedge] (23.center) to (21.center);
		\draw [style=boxedge] (21.center) to (24.center);
		\draw [style=boxedge] (24.center) to (22);
		\draw [style=directed] (26) to (25);
		\draw [style=right arcout] (25) to (10);
		\draw [style=boxedge] (33) to (29.center);
		\draw [style=boxedge] (29.center) to (40.center);
		\draw [style=boxedge] (40.center) to (32.center);
		\draw [style=boxedge] (32.center) to (33);
		\draw [style=directed] (30) to (39);
		\draw [style=right arcout] (39) to (31);
		\draw [style=boxedge] (37) to (34.center);
		\draw [style=boxedge] (34.center) to (28.center);
		\draw [style=boxedge] (28.center) to (38.center);
		\draw [style=boxedge] (38.center) to (37);
		\draw [style=right arcin] (36) to (41);
		\draw [style=directed] (41) to (30);
		\draw [style=boxedge] (46) to (49.center);
		\draw [style=boxedge] (49.center) to (43.center);
		\draw [style=boxedge] (43.center) to (47.center);
		\draw [style=boxedge] (47.center) to (46);
		\draw [style=right arcin] (44) to (48);
		\draw [style=boxedge] (54) to (51.center);
		\draw [style=boxedge] (51.center) to (55.center);
		\draw [style=boxedge] (55.center) to (53.center);
		\draw [style=boxedge] (53.center) to (54);
		\draw [style=directed] (50) to (52);
		\draw [style=directed] (48) to (56);
		\draw [style=right arcout] (56) to (50);
		\draw [style=directed] (45) to (57);
		\draw [style=right arcout] (57) to (50);
		\draw [style=directed] (60) to (10);
		\draw [style=right arcin] (27) to (11);
		\draw [style=directed] (13) to (11);
		\draw [style=directed] (35) to (30);
		\draw [style=uright arcout] (8) to (63.center);
		\draw [style=right arcin] (63.center) to (60);
		\draw [style=boxedge] (71) to (65.center);
		\draw [style=boxedge] (65.center) to (69.center);
		\draw [style=boxedge] (69.center) to (77.center);
		\draw [style=boxedge] (77.center) to (71);
		\draw [style=directed] (73) to (68);
		\draw [style=directed] (66) to (76);
		\draw [style=boxedge] (72) to (75.center);
		\draw [style=boxedge] (75.center) to (70.center);
		\draw [style=boxedge] (70.center) to (64.center);
		\draw [style=boxedge] (64.center) to (72);
		\draw [style=directed] (74) to (67);
		\draw [style=directed] (67) to (66);
		\draw [style=uright arcout] (68) to (78.center);
		\draw [style=right arcin] (78.center) to (66);
	\end{pgfonlayer}
\end{tikzpicture}